\documentclass{amsart}
\usepackage{preprint}

\title{Minimal Model of Ginzburg Algebras}
\author{Stephen Hermes}
\address{Mathematics Department, Wellesley College, Wellesley, MA 02418}
\email{shermes@wellesley.edu}
\thanks{Partially supported by NSA Grant \#H98230-13-1-0247}
\date{\today}

\begin{document}

\maketitle

\begin{abstract}
We compute the minimal model for Ginzburg algebras associated to acyclic quivers $Q$. In particular, we prove that there is a natural grading on the Ginzburg algebra making it formal and quasi-isomorphic to the preprojective algebra in non-Dynkin type, and in Dynkin type is quasi-isomorphic to a twisted polynomial algebra over the preprojective with a unique higher $A_\infty$-composition. To prove these results, we construct and study the minimal model of an $A_\infty$-envelope of the derived category $\Db(Q)$ whose higher compositions encode the triangulated structure of $\Db(Q)$.
\end{abstract}


\section*{Introduction}

Calabi-Yau categories have been enjoying a growing interest in geometry, representation theory, and mathematical physics. These are triangulated $\k$-categories (where $\k$ is an algebraically closed ground field) equipped with bifunctorial isomorphisms $$\Hom(X,Y)\cong D\Hom(Y,S^nX)$$ where $D$ denotes $\k$-linear dual and $S$ is the suspension functor of the triangulated category. Here, $n$ is a fixed integer called the \emph{Calabi-Yau dimension}. Such categories have appeared in diverse areas of mathematics in many different guises. They have appeared as everything from derived categories of sheaves on noncommutative spaces, to categories of $D$-branes in string theory, and to cluster mutation in the representation theory of finite dimensional algebras.

The eponymous instance of such categories comes from algebraic geometry: Serre duality implies that the derived category of coherent sheaves for an $n$-dimensional Calabi-Yau variety $X$ is an $n$-Calabi-Yau category. In some sense this is the general situation. By adopting the categorical approach to noncommutative geometry in the sense of Bondal \cite{Bondal}, one should view an arbitrary Calabi-Yau category as the derived category of coherent sheaves of some noncommutative Calabi-Yau space. 

Dimension three is of particular interest. String theory requires an extra six spatial dimensions taking the form of a complex Calabi-Yau 3-fold. Three-Calabi-Yau categories are an essential ingredient in Kontsevich's homological picture of mirror symmetry \cite{KontsevichHMS}. The motivic DT-invariants of Kontsevich and Soibelman are framed as invariants of 3-Calabi-Yau categories \cite{KontsevichSoibelman}. In the direction of representation theory, 3-Calabi-Yau categories appear in Amiot's construction of cluster categories \cite{Amiot}, extending the construction of Buan, Marsh, Reiten, Reineke, and Todorov, which themselves are 2-Calabi-Yau \cite{BMRRT}.

In his seminal work, \cite{Ginzburg} Ginzburg introduced a class of differential graded algebras called Calabi-Yau algebras which have the remarkable property that their derived categories are Calabi-Yau triangulated categories. For dimension three, one can associate a Ginzburg algebra to an arbitrary quiver with potential in the sense of Derksen, Weyman, and Zelevinsky \cite{DWZ}. Conversely, Keller (and Van den Bergh in the appendix) proved in \cite{Keller5} that the Ginzburg algebra is 3-Calabi-Yau while Van den Bergh proved a converse result in \cite{VanDenBergh}.

Despite their growing mathematical importance, surprisingly little has been done in the study of Ginzburg algebras from the point of view of quiver representations. In this paper we study the $A_\infty$-structure of the Ginzburg algebras associated to acyclic quivers and compute their minimal models. In particular, we give a complete description of the $A_\infty$-structure of the minimal models for all acyclic quivers.

\subsection{Results} Fix a ground field $\k$, which we will assume to be algebraically closed and characteristic 0. Throughout, all categories are assumed to be $\k$-categories, all functors are $\k$-linear, and all unadorned tensor products are over $\k$. For a category $\cat C$ denote by $\cat C(X,Y)$ the space of morphisms from $X$ to $Y$ in $\cat C$.

Recall a \emph{quiver} is a tuple $Q=(Q_0,Q_1,s,t)$ where $Q_0$ is the \emph{vertex set}, $Q_1$ is the \emph{arrow} set, and $s,t:Q_1\to Q_0$ are respectively the \emph{source} and \emph{target} functions. A \emph{path} in $Q$ is a sequence $\alpha_1\alpha_2\cdots\alpha_n$ of arrows with $t(\alpha_i)=s(\alpha_{i+1})$ for each $1\leq i<n$. 

\begin{definition}
The Ginzburg algebra $\Gamma_Q$ associated to an acyclic quiver $Q$ is the dg path algebra $\k\hater{Q}$ where $\hater{Q}$ is the graded quiver obtained from $Q$ as follows:

\begin{enumerate}
\item the arrows of the original quiver have degree $0$,
\item for each $\alpha:i\to j$ in $Q$ adjoin a reversed arrow $\alpha^*:j\to i$ having degree $1$,
\item and for each vertex $i\in Q_0$ adjoin a loop $t_i$ at the vertex $i$ of degree $2$.
\end{enumerate}

The differential $d$ on $\Gamma_Q$ is given by:
\begin{enumerate}
\item $d\alpha=d\alpha^*=0$ for $\alpha\in Q_1$,
\item and $dt_i=\rho_i$ where $$\rho_i=\sum_{\substack{\alpha:i\to j \\ \text{in } Q}}\alpha\alpha^*-\sum_{\substack{\beta:j\to i \\ \text{in } Q}}\beta^*\beta.$$
\end{enumerate}
\end{definition}

\begin{remark}
The restriction to acyclic quivers is not necessary in the above definition provided one also chooses a \emph{potential} $W$ on $Q$, i.e., a linear combination of cycles in $Q$ considered up to cyclic equivalence; for such $(Q,W)$ the formula for $d\alpha^*$ involves terms in $W$. However acyclic quivers admit only the trivial potential and so we take the simplified definition above in our setup.
\end{remark}

We will want to view the Ginzburg algebra as bigraded where the arrows $\alpha$ have bidegree $(0,0)$, the $\alpha^*$ have bidegree $(1,0)$ and the loops $t_i$ have bidegree $(1,1)$. The usual grading of the Ginzburg algebra is given by total degree of this bigrading.

\begin{theoremA*}[Corollary \ref{cor:GinzburgNonDynkin}]
If $Q$ is non-Dynkin then $\Gamma_Q$ is formal, the homology $H_*\Gamma_Q$ is concentrated in degree 0, and is isomorphic to the preprojective algebra $\Lambda_Q$. 
\end{theoremA*}

\begin{remark}
We emphasize that the formality of $\Gamma_Q$ in the above theorem is with respect to the second component of the bigrading, and \emph{not} total degree.
\end{remark}

In the Dynkin setting, the minimal model is more subtle. The Nakayama functor $\nu$ of $\Mod\text{-}\k Q$ can be extended to an automorphism of the preprojective algebra $\Lambda_Q$, and the underlying associative algebra of the minimal model is a polynomial algebra over $\Lambda_Q$ twisted by $\nu$ (see Section \ref{sec:Preprojective}). 

\begin{theoremB*}[Theorem \ref{thm:TwistedDerivedTranslationAlgebra}]
If $Q$ is Dynkin, then there is an $A_\infty$-structure $(\mu_n)_{n\geq 2}$ on the twisted polynomial algebra $\Lambda_Q^\nu[u]$ making it a minimal model of $\Gamma_Q$. Moreover, this $A_\infty$-structure is $u$-invariant, generates $\Lambda_Q^\nu[u]$ over $\Lambda_Q$, and $\mu_n=0$ for $n\neq 2,3$.
\end{theoremB*}

The higher composition $\mu_3$ can be explicitly computed using the triangulated structure of the bounded derived category $\Db(Q)$. In general we find that the $A_\infty$-structure of the Ginzburg algebra $\Gamma_Q$ is intimately related to the triangulated structure of $\Db(Q)$:

\begin{theoremC*}[Theorem \ref{thm:GinzburgDerivedTranslationAlgebra}]
If $Q$ is acyclic, there is a quasi-isomorphism of algebras $$\Gamma_Q\to\bigoplus_{n\geq 0}\Pdg(Q)(\k Q,\tau^{-n}\k Q)$$ where $\Pdg(Q)$ is the (dg) category of bounded complexes of projective $\k Q$-modules and $\tau$ is the Auslander-Reiten translate.
\end{theoremC*}

This theorem is proven by constructing a differential graded Galois $\Z$-covering $\tilder{\Gamma}_Q$ of the algebra $\Gamma_Q$, together with a quasi-fully-faithful functor $R:\GG\to\Pdg(Q)$ from the path category $\GG$ of $\tilder{\Gamma}_Q$. In Section \ref{sec:Covering} we describe how to recover $\Gamma_Q$ in terms of the cover $\tilder{\Gamma}_Q$, providing the quasi-isomorphism of the Theorem. 

The derived category $\Db(Q)$ can be viewed as the zeroth homology of the dg category $\Pdg(Q)$, and the total homology category $H_*\Pdg(Q)$ is equivalent to the orbit category $\Db(Q)^\Z$ of $\Db(Q)$ under the $\Z$-action induced by the shift functor. This determines a minimal $A_\infty$-structure on the category $\Db(Q)^\Z$ by homotopy perturbation. The theorem above then reduces the computation of the minimal model of $\Gamma_Q$ to the computation of the $A_\infty$-structure of the category $\Db(Q)^\Z$. 

\begin{theoremD*}[Theorem \ref{thm:DerivedCategoryAinfty}]
If $Q$ is acyclic, the $A_\infty$-structure $(\mu_n)_{n\geq 2}$ has the following properties:
\begin{enumerate}
\item the compositions $\mu_n$ are equivariant with respect to the canonical degree 1 maps $s_X:X\to X[1]$,
\item if $X\xrightarrow{f}Y\xrightarrow{g}Z\xrightarrow{h}X[1]$ is a non-split triangle in $\Db(Q)$ then $\mu_3(h,g,f)=s_X$,
\item $\mu_n=0$ for $n\neq 2,3$.
\end{enumerate}
\end{theoremD*}

\subsection{Outline of the paper} We begin in Section \ref{sec:DGA} by recalling the notions of differential graded and $A_\infty$-structures on algebras and categories. In particular, we recall the Homotopy Transfer Theorem (Theorem \ref{thm:HomotopyTransfer}) and give an explicit formula of the transferred $A_\infty$-structure in terms of planar binary rooted trees.

Next in Section \ref{sec:Covering} we next adapt the covering theory of Bongartz-Gabriel \cite{BongartzGabriel} to accommodate dg and $A_\infty$-structures using techniques from and Bautista-Liu \cite{BautistaLiu}. We define section subcategories and use these to recover categories from their covers. This will be used extensively in the proofs of Theorems B and C.

In Section \ref{sec:DerivedCategories} we construct an explicit model of an $A_\infty$-structure on the orbit category $\Db(Q)^\Z$. In order to apply the Homotopy Transfer Theorem, we need to construct a suitable small replacement. This procedure is surprisingly subtle, and requires us to develop the notion of a pseudo-skeleton. With this technology in place we give the proof of Theorem D.

The proofs of Theorems A and C are given in Section \ref{sec:Ginzburg}. This section begins with a computation of the $A_\infty$-structure of the algebra $U(Q)=\bigoplus_{n\geq 0}\Pdg(Q)(\k Q,\tau^{-n}\k Q)$ using Theorem D. The quasi-isomorphism of Theorem C is then constructed using the dg covering theory developed in Section \ref{sec:Covering}.

Finally in Section \ref{sec:Dynkin} we specialize to the case where $Q$ is Dynkin and prove Theorem B. In order to do so, we briefly develop twisted polynomial algebras, and introduce the Nakayama involution $\nu$ on the preprojective algebra $\Lambda_Q$ of $Q$. The proof of Theorem B is given by constructing an isomorphism $\Lambda_Q\xrightarrow{\sim} U(Q)$  using covering theory.

	
\section{Differential graded and $A_\infty$-structures}\label{sec:DGA} 

To fix notation, we recall the theory of graded and differential graded algebras, modules, and categories as well as and $A_\infty$-algebras and categories. For more detailed accounts we direct the reader to \cite{KellerDG, Keller1}. 

\subsection{Differential graded categories} 


A \emph{graded $\k$-vector space} is a $\k$-vector space $V$ with a fixed direct sum decomposition $V=\bigoplus_{n\in\Z} V_n$. Elements of the subspace $V_n$ are said to be \emph{homogeneous} of degree $n$, and write $|v|=n$ for $v\in V_n$. 

A $\k$-linear map $f:V\to W$ between graded spaces is \emph{homogeneous} of degree $n$ if $f(V_p)\subseteq W_{p+n}$ for every $p\in\Z$. Graded spaces form a category $\Cgr(\k)$ with morphisms $$\Cgr(\k)(V,W)=\bigoplus_{n\in\Z}\Cgr_n(\k)(V,W)$$ 
where $\Cgr_n(\k)(V,W)$ denotes the space of homogeneous degree $n$ linear maps $f:V\to W$. A degree 0 homogeneous graded map will be called a \emph{chain map}.

The \emph{shift} of a graded space $V$ is the graded space $V[1]$ with $V[1]_n=V_{n-1}$. In particular, if $V$ is concentrated in degree $n$, the shift $V[1]$ is concentrated in degree $n+1$. Similarly, given a homogeneous linear map $f:V\to W$ of degree $n$, one defines $f[1]:V[1]\to W[1]$ by $f[1]=(-1)^nf$. In this way, shift $[1]$ determines an automorphism of $\Cgr(\k)$.

A \emph{graded category} is a category $\cat C$ enriched over the category of graded spaces and chain maps. That is, for any two objects $X$ and $Y$ of $\cat C$ the space of morphisms $$\cat C(X,Y)=\bigoplus_{n\in\Z}\cat C_n(X,Y)$$ is a graded $\k$-vector space, and the composition map $$\cat C(Y,Z)\otimes\cat C(X,Y)\to\cat C(X,Z)$$ is a chain map. A \emph{graded functor} is a functor $F:\cat C\to\cat D$ so that for any two objects $X$ and $Y$ of $\cat C$ the induced map $F_{XY}:\cat C(X,Y)\to\cat D(FX,FY)$ is a chain map. A \emph{contravariant} graded functor is a graded functor $F:\cat C^{\op}\to\cat D$. 
The opposite category $\cat C^{\op}$ is endowed with the composition $g\bullet f=(-1)^{|f||g|}f\circ g$ for homogeneous $f$ and $g$.

A \emph{differential} on a graded space $V$ is a degree $-1$ linear map $d:V\to V$ satisfying $d_V^2=0$. A pair $(V,d_V)$ where $d_V$ is a differential on $V$ is a \emph{differential graded} (dg for short) space. Denote by $\Cdg(\k)$ the full subcategory of $\Cgr(\k)$ whose objects are differential graded spaces. Note that we impose no compatibility between the morphisms of $\Cdg(\k)$ and differentials. Homogeneous degree 0 graded maps $f:V\to W$ such that $d_W\circ f= f\circ d_V$ are called \emph{chain maps}.

Let $A$ be a \emph{differential graded} algebra, i.e., an associative algebra object in $\Cdg(\k)$. The subcategory $\Cdg(A)$ of $\Cdg(\k)$ whose objects are dg $A$-modules and whose morphisms are $A$-linear graded maps is closed under tensor products and shifts. Again, there is no imposed compatibility between the morphisms of $\Cdg(A)$ and differentials. Homogeneous degree 0 morphisms commuting with differentials are again called \emph{chain maps}.

A \emph{differential graded category} is a category $\cat A$ enriched over the category of dg spaces and chain maps. More explicitly, for any two objects $X$ and $Y$ of $\cat A$, the set of morphisms $$\cat A(X,Y)=\bigoplus_{n\in\Z}\cat A_n(X,Y)$$ is a graded space and endowed with a differential $d_{XY}$ of degree $-1$. Moreover for any three objects $X$, $Y$, and $Z$ the composition map $$\cat A(Y,Z)\otimes\cat A(X,Y)\to\cat A(X,Z)$$ is a chain map, and so for any two morphisms $f:X\to Y$ and $g:Y\to Z$ with $g$ homogeneous, the Leibniz law $$d_{XZ}(g\circ f)=d_{YZ}(g)\circ f +(-1)^{|g|}g\circ d_{XY} (f)$$ is satisfied. A \emph{dg functor} is a functor $F:\cat A\to\cat B$ so that for any two objects $X$ and $Y$ of $\cat A$ the induced map $$F_{XY}:\cat A(X,Y)\to\cat B(FX,FY)$$ is a chain map. A \emph{contravariant} dg functor is a dg functor $F:\cat A^{\op}\to\cat B$.

\begin{examples}
\begin{enumerate}
\item The category $\Cdg(\k)$ is a dg category, when the morphism space $\Cdg(\k)(V,W)$ is equipped with the \emph{commutator differential} $$d_{VW}(f)=d_W\circ f-(-1)^{|f|}f\circ d_V$$ where $d_V$ and $d_W$ are the differentials of the dg spaces $V$ and $W$.
\item For a dg algebra $A$, $\Cdg(A)$ is a full dg subcategory of $\Cdg(\k)$.
\item Any full subcategory of a dg category is a dg category as well. In particular, the full category $\Pdg(A)$ of $\Cdg(A)$ whose objects are bounded chain complexes of projective modules is a dg category.
\item The (graded) opposite category of a dg category $\cat A$ is a dg category in the evident manner. 
\end{enumerate}
\end{examples}

\subsection{Underlying, homology and homotopy categories} 

There are several auxiliary categories which can be attached to a dg category.
\begin{definition}
\begin{enumerate}
\item The \emph{underlying category} of $\cat A$ is the category $Z_0\cat A$ whose objects are the same as the objects of $\cat A$, but with morphisms $Z_0\cat A(X,Y)$ given by the 0-cycles of $\cat A(X,Y)$.

\item The \emph{homology} of $\cat A$ is the graded category $H_*\cat A$ whose objects are the same as the objects of $\cat A$, but morphism spaces $$(H_*\cat A)(X,Y)=H_*\cat A(X,Y)$$ where $H_*\cat A(X,Y)=\ker d_{XY}/\im d_{XY}$ is the homology of the dg space $\cat A(X,Y)$.

\item The \emph{homotopy category} of $\cat A$ is the degree 0 subcategory $H_0\cat A$ of the homology $H_*\cat A$. Equivalently, it is the quotient of the underlying category $Z_0\cat A$ by the ideal $B_0\cat A$ of 0-boundaries.
\end{enumerate}
\end{definition}

All of these constructions are functorial in the category of (small) dg categories, i.e., a dg functor $\cat A\to\cat B$ induces (graded) functors between the underlying, homology, and homotopy categories.

If $A$ is a dg algebra, then The 0-cycles of $\Cdg(A)(X,Y)$ are the degree 0 maps $f:X\to Y$ so that $d_{XY}(f)=d_Y\circ f-f\circ d_X=0$, i.e, the chain maps. The morphism spaces $H_0\Cdg(A)(X,Y)$ in the homotopy category of $\Cdg(A)$ are just the spaces of homotopy classes of chain maps $f:X\to Y$. Indeed, the morphisms in $B_0\cat A(X,Y)$ are of the form $d_{XY}(s)=d_Y\circ s+s\circ d_X$ for some degree $-1$ morphism $s:X\to Y$. By a theorem of Happel \cite{Happel}, the derived category $\Db(A)$ is equivalent to $Z_0\Pdg(A)$ modulo homotopy, and hence is equivalent to $H_0\Pdg(A)$.

\subsection{$A_\infty$-Algebras and categories}

An \emph{$A_\infty$-algebra} is a graded space $A$ together with a collection of degree $n-2$ maps $\mu_n:A^{\otimes n}\to A$ for $n\geq 1$ satisfying the relation 
\begin{equation}\label{eqn:Ainfinity}
0=\sum_{\substack{p+q+r=n \\ p,r\geq 0, \ q\geq 1}}(-1)^{p+qr}\mu_{p+1+r}\circ(id^{\otimes p}\otimes\mu_q\otimes id^{\otimes r})
\end{equation}
for each $n$. The maps $\mu_n$ are called the \emph{higher products} or \emph{$A_\infty$-structure maps} of $A_\infty$-algebra $A$.

A \emph{morphism} $f:A\to B$ of $A_\infty$-algebras $(A,\mu_n)$ and $(B,\nu_n)$ is a collection of graded maps $f_n:A^{\otimes n}\to B$ for $n\geq 1$ of degree $n-1$ satisfying the identities 
\begin{equation*}
\sum_{\substack{p+q+r=n \\ p,r\geq 0, \ q\geq 1}}(-1)^{p+qr}f_{p+1+r}\circ (id^{\otimes p}\otimes\mu_q\otimes id^{\otimes r})=\sum_{\substack{1\leq d\leq n \\ i_1+\cdots+i_d=n}}(-1)^D\nu_d\circ (f_{i_1}\otimes\cdots\otimes f_{i_d})
\end{equation*}
for each $n$ where $D=(d-1)(i_1-1)+(d-2)(i_2-1)+\cdots+2(i_{d-2}-1)+(i_{d-1}-1)$.

An $A_\infty$-algebra $A$ is in particular a dg space, so we can form the homology $H_*A=\ker d_A/\im d_A$. The structure map $\mu_2$ of $A$ descends to an associative multiplication on $H_*A$, making it into a graded algebra. For an $A_\infty$-morphism $f:A\to B$, the map $f_1$ is a chain map and so induces a map $(f_1)_*:H_*A\to H_*B$ in homology. A morphism $f$ is an \emph{$A_\infty$-quasi-isomorphism} if $f_1$ is a quasi-isomorphism, i.e., the induced map $(f_1)_*$ is an isomorphism, and $f$ is \emph{strict} if $f_n=0$ for $n\neq 1$. The \emph{identity morphism} of an $A_\infty$-algebra $A$ is the strict morphism $id_A$ whose $n=1$ component is the identity on $A$.

Similarly one can extend the notion of a homotopy between dg morphisms to the $A_\infty$-setting. An advantage of studying homotopy theory in the category of $A_\infty$-algebras is that $A_\infty$-quasi-isomorphisms are invertible up to homotopy\textemdash that is, if $f:A\to B$ is an $A_\infty$-quasi-isomorphism, it necessarily admits a homotopy inverse.

\subsection{Minimal models}

An $A_\infty$-algebra $(A,\mu_n)$ is \emph{minimal} if $\mu_1=0$. An $A_\infty$-algebra $B$ is a \emph{minimal model} for $A$ if $B$ is minimal and there is an $A_\infty$-quasi-isomorphism $f:A\to B$. The $A_\infty$-algebra $A$ is \emph{formal} if it admits a minimal model whose higher multiplications vanish for $n\geq 3$.

If $A$ is any $A_\infty$-algebra, the homology $H_*A$ is minimal. The following theorem of Kad\-ei\v{s}\-vili makes $H_*A$ into a minimal model of $A$ in the case that $A$ is a dg algebra.

\begin{theorem}[Kadei\v{s}vili \cite{Kadeishvili}]\label{thm:Kadeishvili} 
Let $A$ be a differential graded algebra. Then there is a minimal $A_\infty$-algebra structure $(\mu_n)_{n\geq 2}$ on the homology $H_*A$ with $\mu_2$ equal the usual multiplication map together with an $A_\infty$-quasi-isomorphism $H_*A\to A$.\\
Moreover, this $A_\infty$-structure on $H_*A$ is unique up to a (non-unique) isomorphism of $A_\infty$-algebras.
\end{theorem} 

In particular, Kad\-ei\v{s}\-vili's Theorem shows that the minimal model of a dg algebra is unique up to $A_\infty$-isomorphism. By abuse of terminology, we will refer to $H_*A$ with the above $A_\infty$-structure as \emph{the} minimal model of $A$, and $(\mu_n)_{n\geq 2}$ as \emph{the} minimal $A_\infty$-structure of $H_*A$.\\ 

Much of the theory of $A_\infty$-structures on algebras can be directly generalized to the categorical setting. These $A_\infty$-categories were first introduced by Fukaya \cite{Fukaya}, and further developed by Seidel \cite{Seidel}.

\begin{definition}
An \emph{$A_\infty$-category} is an $A_\infty$-algebra with multiple objects. More precisely, an $A_\infty$-category $\cat A$ consists of the data of a class of objects, a graded space $\cat A(X,Y)$ for every pair of objects $X$ and $Y$ of $\cat A$, together with $n$-airy compositions
$$\mu_n^{X_0,\dots, X_n}:\cat A(X_{n-1},X_n)\otimes\cdots\otimes\cat A(X_1,X_2)\otimes\cat A(X_0,X_1)\to \cat A(X_0,X_n)$$ satisfying relations analogous to \eqref{eqn:Ainfinity}.
\end{definition}

Many of the theorems and techniques available for $A_\infty$-algebras pass with minimal modification to the categorical setting. In particular, every (small) dg category $\cat A$ admits its homology category $H_*\cat A$ as a minimal model.

\subsection{Homotopy Transfer Theorem}\label{sec:HTT}

In order to explicitly construct the minimal $A_\infty$-structure from Kad\-ei\v{s}\-vili's Theorem we will need the Homotopy Transfer Theorem as presented in Loday-Vallette \cite{LodayVallette}. To state the Homotopy Transfer Theorem concretely, we will need some elementary facts about planar binary rooted trees.

\begin{definition}
A \emph{planar binary rooted $n$-tree} (PBR $n$-tree for short) is a trivalent planar graph $T$ with $n+1$ external edges, one of which is distinguished and called the \emph{root edge}. The other $n$ external edges are called the \emph{leaves} of $T$. We think of an external edge as being adjacent to only one vertex; all edges adjacent to two vertices are internal edges.
\end{definition}

Deleting a small neighborhood of a vertex $v$ of $T$ breaks $T$ into three connected components: the component containing the root edge, and the \emph{left} and \emph{right} \emph{subtrees} of $T^-(v)$ and $T^+(v)$ of $T$ at $v$. The left/right subtrees of $T$ are PBR trees with planar embedding induced from that of $T$ and root edge given by the edge formerly connected to $v$.

There is a standard bijection between PBR trees and $(231)$-avoiding permutations. Given a PBR tree $T$, one can construct a $(231)$-avoiding permutation $\sigma_T$ as follows. Choose an embedding of $T$ in the plane $\R^2$ so that for each vertex of $T$ the left subtree lies below the right subtree, and no two vertices have the same $(x,y)$-coordinates. This induces two total orderings on the vertices: a horizontal ordering and a vertical ordering, and hence horizontal and vertical order-preserving bijections $h,v:T_0\to\{1,2,\dots,n\}$ from the set of vertices $T_0$ of $T$. Then $\sigma_T(i)=v(h^{-1}(i))$. 

\begin{definition}
The \emph{sign} of a PBR tree $T$ is $\sgn(T)=\sgn(\sigma_T)$ where $\sigma_T$ is the corresponding permutation.
\end{definition}

Recall that a dg space $V$ is a \emph{homotopy retract} of a dg space $A$ if there is a diagram
\[
\xymatrix{A\ar@/^/[r]^{q}\ar@(ul,dl)[]_\phi & V\ar@/^/[l]^{j}}
\]
such that $qj=id_V$ and $\phi:id_A\simeq jq$ is a homotopy equivalence. Every dg space $A$ admits its homology $H_*A$ as a homotopy retract (\emph{cf. e.g.}, \cite{LodayVallette} Lemma 9.4.7). 

Let $V$ be a homotopy retract of a dg algebra $A$. For each PBR $n$-tree $T$ define a map $\mu_T:V^{\otimes n}\to V$ as follows: place the map $j$ on the leaves of $T$, $q$ on the root, $\phi$ on each internal edge, and the multiplication $\mu_A$ on each (internal) vertex. The maps are directed towards the root of $T$. The map $\mu_T$ has degree $n-2$ since there is a copy of the degree 1 map $\phi$ for each of the $n-2$ internal edges of $T$.

A more formal construction of the maps $\mu_T$ for can be provided by induction on PBR trees. To do so, we define degree $n-2$ maps $\nu_T:A^{\otimes n}\to A$ such that $\mu_T=q\circ\nu_T\circ j^{\otimes n}$ for each PBR tree $T$.

\begin{construction}\label{constr:NuMaps}
For the unique PBR 2-tree $Y$, $\nu_Y$ is simply the multiplication map $\mu$.
If $T$ is an arbitrary PBR tree, deleting a small neighborhood of the internal vertex adjacent to the root edge gives two smaller PBR trees $T^-$ and $T^+$ (the left and right subtrees). Define $$\nu_T=\mu\circ(\nu'_{T^-}\otimes\nu'_{T^+})$$ where $\nu'_T=\phi\circ\nu_T$ if $T\neq\varnothing$ and is $id_A$ otherwise. (Recall that in our convention, leaves are only adjacent to one edge.)
\end{construction}

\begin{theorem}[Homotopy Transfer]\label{thm:HomotopyTransfer}
The maps $$\mu_n=\sum_{T\in PBR_n}\sgn(T)\mu_T$$ endow $V$ with the structure of an $A_\infty$-algebra. Moreover, the maps $q$ and $j$ can be extended to $A_\infty$-quasi-isomorphisms, and $\phi$ can be extended to an $A_\infty$-homotopy.
\end{theorem}

\begin{remark}
The $A_\infty$-extensions of $q$ and $\phi$ are due to Markl \cite{Markl}.
\end{remark}

The construction of the minimal model of Kad\-ei\v{s}\-vili's Theorem given by the Homotopy Transfer Theorem readily generalizes to the categorical setting. Given a small dg category $\cat A$, one needs to choose homotopy retractions of $\cat A(X,Y)$ for each pair of objects of $X$ and $Y$. Associated to each PBR $n$-tree $T$, one constructs homogeneous maps
$$\mu_T^{X_0,\dots, X_n}:\cat A(X_{n-1},X_n)\otimes\cdots\otimes\cat A(X_1,X_2)\otimes\cat A(X_0,X_1)\to \cat A(X_0,X_n)$$
analogously to \ref{constr:NuMaps}, and the higher compositions 
$$\mu_n^{X_0,\dots, X_n}=\sum_{T\in PBR_n}\sgn(T)\mu_T^{X_0,\dots, X_n}$$
define an $A_\infty$-structure on $H_*\cat A$.


\section{Differential graded covering theory}\label{sec:Covering}

The goals of this section are the twofold. First, we adapt the dictionary between algebras and (small) categories to the differential graded setting. Moreover, we modify the covering theory of Bongartz and Gabriel \cite{BongartzGabriel} to accommodate dg structures. Secondly we discuss how to recover a dg/graded algebra from a covering and show that this recovery process is compatible with passing to quotients and homology.

We begin by recalling the relationships between quivers, algebras, and locally finite dimensional categories.

\subsection{Path categories}

The \emph{associated algebra} of a small category $\cat C$ is the algebra $\k[\cat C]$ with underlying vector space $\bigoplus_{X,Y}\cat C(X,Y)$ where the direct sum is taken over all pairs of objects $X$, $Y$ in $\cat C$. The product in $\k[\cat C]$ is given by composition if the elements are composeable, and is $0$ otherwise. 

If $\cat A$ is a small dg (resp. graded) category, then the associated algebra $\k[\cat A]$ inherits the a dg (resp. graded) algebra with $d_{\k[\cat A]}=\sum d_{XY}$. The associated algebra construction is compatible with homology in the sense that $H_*\k[\cat A]=\k[H_*\cat A]$.

A functor $F:\cat C\to\cat D$ between small categories induces a (not necessarily unital) algebra homomorphism between the associated algebras whose components are the canonical maps $$F_{XY}:\cat C(X,Y)\to\cat D(FX,FY)$$ which is a dg (resp. graded) homomorphism if $f$ is a dg (resp. graded) functor.

The \emph{path category} of a quiver $Q$ is the small category $\cat C_Q$ whose objects are the vertices of $Q$ and whose morphism spaces are given by $\cat C_Q(i,j)=e_j\k Qe_i$. Composition is given by concatenation of paths. It is immediate that the algebra $\k[\cat C_Q]$ associated to the small category $\cat C_Q$ is isomorphic to the path algebra $\k Q$. When $Q$ is acyclic, the category $\cat C_Q$ is \emph{locally finite dimensional} in the sense that:

\begin{enumerate}
\item distinct objects are non-isomorphic, 
\item morphism spaces are finite dimensional, 
\item and all endomorphism algebras are one dimensional.
\end{enumerate}

In fact all such categories arise as the path category of a quiver modulo an ideal \cite{BongartzGabriel}.

\begin{definition}
A \emph{differential graded quiver} is a quiver $Q$ whose arrow set is $\Z$-graded together with a degree $-1$ function $d:Q_1\to\k Q_2$ between $Q_1$ and the span of the paths of length 2, such that $d$ is compatible with the source and target maps in the sense that $d(\alpha)\in e_{s(\alpha)}\k Qe_{t(\alpha)}$ for any arrow $\alpha$. 

A \emph{morphism} $f:(Q,d)\to(Q',d')$ of dg quivers is a quiver morphism satisfying $d'\circ f=f_*\circ d$ where $f_*:\k Q_2\to\k Q'_2$ is the induced $\k$-linear map.
\end{definition}

The differential $d$ of a dg quiver $Q$ can be extended by the Leibniz law to a derivation of the path algebra $\k Q$, making it into a dg algebra, and a dg quiver morphism induces a dg homomorphism between path algebras. Likewise, the path category $\cat C_Q$ of a dg quiver is naturally a dg category. One readily verifies that for such a quiver, the natural isomorphism $\k[\cat C_Q]\cong\k Q$ is a dg isomorphism.  

\begin{remark}
The path algebra and path category of a graded quiver $Q$ are naturally bigraded: one grading is induced by the grading on $Q$ while the other grading is by path length. The Leibniz law guarantees that the differential on is automatically homogeneous of degree $-1$ with respect to the induced grading. The condition that the differential of a dg quiver take image in $\k Q_2\subset\k Q$ ensures that the differential is homogeneous of degree $+1$ with respect to the length grading.
\end{remark}

\subsection{Modules}

Given a category $\cat A$, a \emph{(right) $\cat A$-module} is a contravariant functor $M:\cat A\to \cat C(\k)$; a \emph{homomorphism} of $\cat C$-modules is a natural transformation $f:M\to N$. Denote by $\cat C(\cat A)$ the category of $\cat A$-modules. For a dg (resp. graded) category, a \emph{dg (resp. graded) $\cat A$-module} is a dg (resp. graded) functor $M:\cat A\to\Cdg(\k)$ (resp. $M:\cat A\to\Cgr(\k)$). Denote by $\Cdg(\cat A)$ (resp. $\Cgr(\cat A)$) the category of dg (resp. graded) $\cat A$-modules.

The categories $\cat C(\k[\cat A])$ and $\cat C(\cat A)$ are naturally equivalent, and likewise for their graded/dg cousins. Hence, we may view $\cat A$-modules and $\k[\cat A]$-modules interchangeably.

\begin{examples}
\begin{enumerate}
\item For any object $X$ of $\cat A$, the functor $\cat A(-,X):\cat A\to\cat C(\k)$ determines a projective $\cat A$-module.
\item For a quiver $Q$, $P_i(-)=\cat C_Q(-,i)$ is the canonical projective module with simple top supported at the vertex $i$.
\end{enumerate}
\end{examples}

\subsection{Ideals} 

If $\cat I$ is an ideal of the small category $\cat C$ then $\k[\cat I]=\bigoplus_{X,Y}\cat I(X,Y)$ is an ideal of the algebra $\k[\cat C]$. The induced algebra homomorphism $\pi_*:\k[\cat C]\to\k[\cat C/\cat I]$ has kernel $\k[\cat I]$ and hence induces an isomorphism
$$\k[\cat C]/\k[\cat I]\xrightarrow{\sim}\k[\cat C/\cat I].$$

If the category $\cat C$ has a graded or dg structure, we say $\cat I$ is a \emph{graded ideal} if each $\cat I(X,Y)$ is generated over $\cat C_0(X,Y)$ by homogeneous elements, and is a \emph{dg ideal} if moreover each of the subspaces $\cat I(X,Y)$ are closed under the differential $d_{XY}$. The quotient category of a dg (resp. graded) category by a dg (resp. graded) ideal remains dg (resp. graded).

In particular for a dg category $\cat A$, the space of boundaries $B_*\cat A$ form a graded ideal of the category of cycles $Z_*\cat A$ and so there is an isomorphism 

\begin{equation}\label{eqn:homologyandk}
H_*\k[\cat A]\cong\k[Z_*\cat A]/\k[B_*\cat A]\xrightarrow{\sim}\k[H_*\cat A].
\end{equation}

Suppose that $\cat A$ is a small $A_\infty$-category. Analogous to the situation of ordinary categories, we can form the \emph{associated $A_\infty$-algebra} $\k[\cat A]=\bigoplus_{X,Y}\cat A(X,Y)$. The structure maps are given by the higher compositions if defined, and are 0 otherwise. 

If $\cat A$ is a small dg category, Kad\-ei\v{s}\-vili's Theorem endows $H_*\cat A$ with the structure of a (small) minimal $A_\infty$-category $A_\infty$-quasi-isomorphic to $\cat A$. Hence the space $\k[H_*\cat A]$ inherits the structure of an $A_\infty$-algebra quasi-isomorphic to $\k[\cat A]$. Applying the isomorphism \eqref{eqn:homologyandk} yields the following:

\begin{lemma}\label{lem:smallAinfty}
The homology $A_\infty$-algebra $H_*\k[\cat A]$ is a minimal model for $\k[\cat A]$, and thus $A_\infty$-isomorphic to $H_*\k[\cat A]$.
\end{lemma}

\subsection{Coverings}

We now recall the notion of a covering of quivers as developed in Bongartz-Gabriel \cite{BongartzGabriel}. A quiver morphism $p:\tilder{Q}\to Q$ is a \emph{covering} if for every vertex $i\in \tilder{Q}_0$ the induced maps
\begin{equation}\label{eqn:QuiverCovering}
p_*:\tilde{s}^{-1}(i)\xrightarrow{\sim} s^{-1}(p(i)) \ \quad \ \text{and} \ \quad \ p_*:\tilde{t}^{-1}(i)\xrightarrow{\sim} t^{-1}(p(i))
\end{equation} 
are bijective. 

Given a covering $p:\tilder{Q}\to Q$, the induced functor $p_*:\cat C_{\tilder{Q}}\to\cat C_Q$ is a covering of locally finite dimensional categories in the sense that the induced maps

\begin{equation}\label{eqn:Covering}
\bigoplus_{j'\in p_*^{-1}(j)}\cat C_{\tilder{Q}}(i,j')\xrightarrow{\sim}\cat C_Q(p_*(i),j) \ \quad \ \text{and} \ \quad \ \bigoplus_{i'\in p_*^{-1}(i)}\cat C_{\tilder{Q}}(i',j)\xrightarrow{\sim}\cat C_Q(i,p_*(j))
\end{equation}
are isomorphisms.

\begin{definition}
A dg quiver morphism $p:(\tilder{Q},\tilder{d})\to (Q,d)$ which is also a covering will be called a \emph{dg covering}. Note that the induced maps \eqref{eqn:Covering} are automatically dg isomorphisms, and the induced functor $p_*:\cat C_{\tilder{Q}}\to\cat C_Q$ is a dg functor.
\end{definition}

We will mainly be interested in coverings arisings from a group action on a quiver $Q$, i.e., a homomorphism $G\to\Aut(Q)$ from $G$ to the group of (orientation preserving) automorphisms of $Q$. A \emph{$G$-quiver} is a quiver with a fixed free $G$-action. For a dg quiver $(Q,d)$, the $G$-action is said to be \emph{differential graded} if it acts by dg quiver automorphisms.

For a $G$-quiver $Q$, the \emph{orbit quiver} $Q/G$ is the quiver with vertex set $(Q/G)_0=Q_0/G$ the set of $G$-orbits of the vertices $Q_0$, and arrow set $(Q/G)_1=Q_1/G$ the set of $G$-orbits of the arrows $Q_1$. The source and target maps $s,t:(Q/G)_1\to (Q/G)_0$ are defined by $$s(G\alpha)=Gs(\alpha) \ \quad \ \text{and} \ \quad \ t(G\alpha)=Gt(\alpha)$$ which are independent of the orbit representatives.

There is a canonical quiver morphism $\pi:Q\to Q/G$ that sends a vertex/arrow to the corresponding orbit. It is universal among quiver morphisms $f:Q\to Q'$ such that $f(gi)=f(i)$ and $f(g\alpha)=f(\alpha)$ for every $i\in Q_0$, $\alpha\in Q_1$, and $g\in G$.

If $(Q,d)$ is a dg quiver with dg $G$-action, the orbit quiver $Q/G$ can be equipped with a differential $d_G$ given by $d_G(G\alpha)=Gd(\alpha)$, and the morphism $\pi:Q\to Q/G$ is a morphism of dg quivers.

\begin{definition}
If $\tilder{Q}$ is a $G$-quiver, a covering $p:\tilder{Q}\to Q$ is said to be a \emph{Galois $G$-cover} if the action of $G$ preserves fibers $p^{-1}(i)$ and is fiberwise transitive.
\end{definition}

The canonical morphism $\pi:Q\to Q/G$ is a Galois $G$-covering, provided that $G$ acts freely.

Bongartz and Gabriel develop their covering theory in the language of locally finite dimensional categories. This notion is not well-adapted to dg covering theory, as is not a homotopy invariant concept. To circumvent this issue, we use Bautista-Liu's covering theory of linear categories \cite{BautistaLiu}.

Recall that an \emph{action} of a group $G$ on a category $\cat C$ is a group homomorphism $G\to\Aut(\cat C)$ where $\Aut(\cat C)$ is the group of $\k$-linear automorphisms of $\cat C$. The $G$ action is \emph{free} if the objects $gX$ and $X$ are non-isomorphic for any $X$ indecomposable and $g\neq 1$. For brevity, a category with a fixed free $G$-action will be called a \emph{$G$-category}. 

The \emph{orbit category} of a $G$-category is the category $\cat C^G$ with the same objects as $\cat C$, and morphism spaces
$$\cat C^G(X,Y)=\bigoplus_{g\in G}(X,gY)$$ with composition determined by $$v\circ u =(gv)\circ u$$ for $u:X\to gY$ and $v:Y\to hZ$ in $\cat C$. There is an evident functor $\pi:\cat C\to \cat C^G$ acting as the identity on objects and morphisms. 

\begin{definition}[Bautista Liu \cite{BautistaLiu} Definition 2.3]\label{def:Stabilization}
A functor $F:\cat C\to\cat D$ is \emph{$G$-stable} if there is a collection $\gamma=\{\gamma^g:F\circ g\xrightarrow{\sim} F\}$ of invertible natural transformations so that 
$$\xymatrix@C=8pt{F(ghX)\ar[rr]^{\gamma^{gh}_X}\ar[dr]_{\gamma^g_{hX}} & & FX \\ & F(hX)\ar[ur]_{\gamma^h_X}}$$ commutes for every object $X$; the collection $\gamma$ is called a \emph{stabilization} of $F$.
\end{definition}

In $\cat C^G$, there are distinguished morphisms $\sigma^g_X\in\cat C^G(gX,X)$ whose $g$-component is $id_{gX}:gX\to gX$, and all others are 0. The morphism $\sigma^g_X$ is clearly invertible, with inverse $\sigma^{g^{-1}}_{gX}=id_{X}:g^{-1}gX\to X$. The collection of natural transformations $\sigma^g$ satisfy the commutativity property of Definition \ref{def:Stabilization} and so the functor $\pi:\cat C\to \cat C^G$ is $G$-stable.

\begin{remark}
If $\gamma$ is a stabilization of $F$ and $\chi:G\to\k^\times$ is any character of $G$, one can twist the stabilization $\gamma$ by $\chi$ to obtain a new stabilization $\gamma^\chi$ with $\gamma^{\chi,g}=\chi(g)\gamma^g$.
\end{remark}

We will frequently need the following well-known proposition (\emph{cf. e.g.} \cite{Keller3}).

\begin{proposition}
The pair $(\pi,\sigma)$ is universal among $G$-stable functors $F:\cat C\to \cat D$ with a fixed stabilization. Precisely, if $F:\cat C\to\cat D$ is $G$-stable with stabilization $\gamma$, there is a unique functor $\bar{F}:\cat C^G\to\cat D$ so that $\gamma^g_X=\bar{F}(\sigma^g_X)$ for each $X$ and $g\in G$ and  
$$\xymatrix@C=8pt{\cat C\ar[rr]^F\ar[dr] & & \cat D \\ & \cat C^G\ar[ur]_{\bar{F}}}$$
commutes up to natural isomorphism.
\end{proposition}

\begin{proof}
The functor $\bar{F}$ is given by sending a morphism $u:X\to gY$ to the composition $\gamma^g_Y\circ Fu$, where $\gamma$ is a $G$-stabilization of $F$. Note that necessarily $\bar{F}(\sigma^g_X)=\gamma_X^g$.
\end{proof}

A $G$-action on a quiver $Q$ induces actions on both the path algebra $\k Q$ and the path category $\cat C_Q$ by $g\cdot u=g_*(u)$ for a path/morphism $u$. If $(Q,d)$ is a dg quiver, the condition of a $G$-action being dg is equivalent to the above actions on $\k Q$ and $\cat C_Q$ commute with differentials.

The induced action on $\cat C_Q$ is free provided the action of $G$ on $Q$ was free. The induced functor $\pi_*:\cat C_Q\to\cat C_{Q/G}$ is manifestly $G$-stable with stabilization $\gamma^g:\pi_*\circ g\xrightarrow{\sim} \pi_*$ given by $\gamma^g_i:Ggi=Gi$, and so there is an induced functor $\cat C_Q^G\to\cat C_{Q/G}$.

\begin{proposition}\label{prop:CategoryGaloisCover}
The induced functor $$\cat C_Q^G\xrightarrow{\sim}\cat C_{Q/G}$$ is an equivalence (but not necessarily an isomorphism).
\end{proposition}

\begin{proof}
It is clear that the functor is surjective on objects, so we only need to show that the maps
$$\cat C_Q^G(i,j)\to\cat C_{Q/G}(Gi,Gj)$$
are isomorphisms. Clearly the above map is surjective. Suppose that two paths $u$ and $v$ ending at $i$ lie in the same $G$  orbit, i.e., $u=gv$ for some $g\in G$. Since $t(u)=t(gv)=gt(v)$ we have $i=gi$, which implies $g=1$ as $G$ acts freely. But then $u=v$ proving injectivity.
\end{proof}

In order to work with Galois coverings between not necessarily locally finite dimensional categories, we need the more general notion of a Galois covering between Krull-Remak-Schmidt categories in the sense of Bautista-Liu.

\begin{definition}[\cite{BautistaLiu} Definition 2.8]\label{def:GaloisCovering}
A functor $F:\cat C\to\cat D$ is a \emph{Galois $G$-covering} if
\begin{enumerate}
\item $F$ is $G$-stable with stabilization $\gamma$ so that the maps

\begin{align}\label{eqn:CoveringIsomorphisms}
\begin{array}{c@{\mskip\thickmuskip}l}
\displaystyle{\bigoplus_{g\in G}\cat C(X,gY)\xrightarrow{\sim}\cat D(FX,FY)} \\
\displaystyle{(u_g)_{g\in G}\mapsto\sum_{g\in G}\gamma^g_Y\circ F(u_g)}
\end{array}
\qquad \qquad 
\begin{array}{c@{\mskip\thickmuskip}l}
\displaystyle{\bigoplus_{g\in G}\cat C(gX,Y)\xrightarrow{\sim}\cat D(FX,FY)} \\
\displaystyle{(v_g)_{g\in G}\mapsto\sum_{g\in G}F(v_g)\circ(\gamma_X^g)^{-1}}
\end{array}
\end{align}

are isomorphisms,
\item $F$ is essentially surjective,
\item $FX$ is indecomposable for any indecomposable $X$,
\item\label{def:GaloisCovering.g} and if $X$, $Y$ are indecomposable with $FX\cong FY$, then there is some $g\in G$ so that $Y=gX$.
\end{enumerate}
The group $G$ is said to be the \emph{Galois group} the covering.
\end{definition}

\begin{examples}
\begin{enumerate}
\item If $\cat C$ is a $G$-category the covering $\pi:\cat C\to\cat C^G$ is a Galois $G$-covering.
\item If $\cat C$ and $\cat D$ are locally finite dimensional categories, then a Galois covering is a Galois covering in the sense of Bongartz-Gabriel. 
\item By virtue of Proposition \ref{prop:CategoryGaloisCover}, the functor $\cat C_Q\to\cat C_{Q/G}$ is a Galois $G$-covering. 
\end{enumerate}
\end{examples}

\begin{definition}
A $G$-action on a category $\cat C$ is \emph{directed} if for 
any two indecomposable objects $X$ and $Y$, there is at most one $g$ so that both $\cat C(gY,X)$ and $\cat C(X,gY)$ are nontrivial.
\end{definition}

\begin{remark}
If the action of $G$ on $\cat C$ is directed, then the group element $g\in G$ as per \eqref{def:GaloisCovering.g} of Definition \ref{def:GaloisCovering} is necessarily unique.
\end{remark}

\subsection{Sections of coverings} 

Given small categories $\cat C$ and $\cat D$ and a Galois $G$-covering $F:\cat C\to\cat D$, an equivalence $\bar{F}:\cat C^G\xrightarrow{\sim}\cat D$ does not in general induce an isomorphism between the associated algebras $\k[\cat C^G]$ and $\k[\cat D]$. Indeed, a necessary condition for the algebras to be isomorphic is that the equivalence $\bar{F}$ gives a bijection between the sets of objects of $\cat C^G$ and $\cat D$. Hence the functor $F$ induces an isomorphism $\k[\cat C^G]\xrightarrow{\sim}\k[\cat D]$ if and only if the equivalence $\bar{F}$ is a categorical isomorphism, i.e., an equivalence that is bijective on object sets. 

We would like a way of recovering the associated algebra $\k[\cat D]$ from $\k[\cat C^G]$, and in particular, recover $\k[\cat C_{Q/G}]$ from $\k[\cat C_Q^G]$ for a quiver $Q$. To this end, we introduce the notion of a section subcategory.

\begin{definition}
A \emph{section} of a Galois $G$-covering $F:\cat C\to\cat D$ is an additive section $S$ of the map $F:\operatorname{Ob}\cat C\to\operatorname{Ob}\cat D$. The \emph{section subcategory} $\cat C^S$ is the full subcategory of $\cat C^G$ generated by the image of $S$.
\end{definition}

\begin{remark}
Any covering functor $F$ between small categories admits a section. Indeed, for $F:\cat C\to\cat D$ to be a covering, the induced map on object sets must be surjective, and so we can choose any additive set theoretic section of this map. 
\end{remark}

\begin{proposition}\label{prop:SectionIsomorphism}
The restriction of the induced functor $\bar{F}:\cat C^G\to\cat D$ to a section subcategory $\cat C^S$ induces an isomorphism $$\cat C^S\xrightarrow{\sim}\cat D$$ of small categories for any section $S$ of $F$. In particular, the induced algebra homomorphism $\k[\cat C^S]\xrightarrow{\sim}\k[\cat D]$ is an isomorphism.
\end{proposition}

\begin{proof}
For every pair of objects $X$, $Y$ of $\cat D$ there is by definition an isomorphism 
$$\bigoplus_{g\in G}\cat C(SX, gSY)\xrightarrow{\sim}\cat D(X,Y)$$ which just so happens to be the canonical map induced by the functor $\bar{F}:\cat C^G\to\cat D$. Extend the association $X\mapsto SX$ to a functor $S:\cat D\to\cat C^G$ by sending a morphism $u:X\to Y$ to the unique element $S(u)$ of $\cat C^G(SX,SY)$ mapping to $u$ under $\bar{F}_{SX,SY}$.
By construction, $\bar{F}\circ S=\Id_{\cat D}$, and the restriction of $S\circ\bar{F}\big|_{\cat C^S}=\Id_{\cat C^S}$, since the objects of $\cat C^S$ are of the form $SX$ for $X$ in $\cat D$.
\end{proof}

\subsection{Quotients and homology}

We now want to study the behavior of coverings and sections under quotienting by ideals and passing to homology. An ideal $\cat I$ of a $G$-category $\cat C$ is \emph{$G$-stable} if $gu\in\cat I(gX,gY)$ for every $u\in\cat I(X,Y)$ and $g\in G$. The $G$-action on $\cat C$ descends to a well-defined $G$-action on the quotient category $\cat C/\cat I$ in the obvious manner when $\cat I$ is $G$-stable.

Let $\cat I^G(X,Y)=\bigoplus_{g\in G}(X,gY)$, which determines an ideal in the orbit category $\cat C^G$.

\begin{lemma}
The categories $\cat C^G/\cat I^G$ and $(\cat C/\cat I)^G$ are naturally equivalent.
\end{lemma}

\begin{proof}
The functor $F:\cat C/\cat I\to\cat C^G/\cat I^G$ is $G$-stable and induces a functor $\bar{F}:(\cat C/\cat I)^G\to C^G/\cat I^G$ by the universal property of orbit categories. The composition $\cat C\to\cat C/\cat I\to(\cat C/\cat I)^G$ is $G$-stable and the induced functor $\cat C^G\to(\cat C/\cat I)^G$ annihilates $\cat I^G$, hence induces a functor $\cat C^G/\cat I^G\to(\cat C/\cat I)^G$ by the universal property of quotient categories. It inverse to $\bar{F}$ by universality.
\end{proof}

\begin{lemma}\label{lem:CoveringIdeals}
Suppose $F:\cat C\to D$ is a Galois $G$-covering and $\cat I$ and $\cat J$ ideals of $\cat C$ and $\cat D$ respectively such that $\cat I$ is $G$-invariant and the isomorphisms \eqref{eqn:CoveringIsomorphisms} restrict to isomorphisms $$\bigoplus_{g\in G}\cat I(X,gY)\xrightarrow{\sim}\cat J(FX,FY).$$ 
Then the induced functor $\bar{F}:\cat C/\cat I\to\cat D/\cat J$ is a Galois $G$-covering.
\end{lemma}

\begin{proof}
The composition $\cat C\to\cat D\to\cat D/\cat J$ annihilates $\cat I$, and so there is an induced functor $\bar{F}:\cat C/\cat I\to\cat D/\cat J$, which we claim is a Galois covering. Since $\cat I$ is $G$-stable, there is an induced $G$-action on $\cat C/\cat I$, and the maps $\gamma^g_Y+\cat I(gY,Y)$ give a $G$-stabilization of $\bar{F}$.

To show that $\bar{F}$ is Galois, we need to show that for any two objects $X$ and $Y$ of $\cat C$ the map 
\begin{equation}\label{eqn:CoveringIdeals}
\bigoplus_{g\in G}(\cat C/\cat I)(X,gY)\to(\cat D/\cat J)(FX,FY)
\end{equation}
is an isomorphism. The map \eqref{eqn:CoveringIdeals} fits into a commutative diagram
$$\xymatrix{0\ar[r] & \bigoplus_{g\in G}\cat I(X,gY)\ar[r]\ar[d] & \bigoplus_{g\in G}\cat C(X,gY)\ar[r]\ar[d]& \bigoplus_{g\in G}(\cat C/\cat I)(X,gY)\ar[r]\ar[d] & 0 \\
0\ar[r] & \cat J(FX,FY)\ar[r] & \cat D(FX,FY)\ar[r] & (\cat D/\cat J)(FX,FY)\ar[r] & 0}$$
where the first and second vertical map are isomorphisms by assumption. By the Five Lemma \eqref{eqn:CoveringIdeals} is an isomorphism.
\end{proof}

The following proposition is immediate.

\begin{proposition}\label{prop:CoveringIdeals}
Given the setup of the previous Lemma and a section $S$ of $F$, the association $S$ is also a section of $\bar{F}$ and the section category $(\cat C/\cat I)^S$ is isomorphic to $\cat C^S/\cat I^S$ where $\cat I^S=\cat I^G\cap\cat C^S$. \qed
\end{proposition}

Suppose now that $F:\cat A\to\cat B$ is a dg $G$-Galois covering, and the stabilization $\gamma^g$ is given by homogeneous cycles in $\cat B$. The action of $G$ descends to an action on $Z_*\cat A$ preserving the ideal $B_*\cat A$ of boundaries, and hence gives a graded $G$-action on the homology category $H_*\cat A$.

\begin{lemma}
In the situation above, the induced homology functor $$H_*F:H_*\cat A\to H_*\cat B$$ is a Galois $G$-covering.
\end{lemma}

\begin{proof}
First consider the restricted map $\bigoplus_{g\in G}Z_*\cat A(X,gY)\to Z_*\cat B(FX,FY)$. It indeed takes image in the cycles of $\cat B$, since the $\gamma^g_Y$ are cycles and $F_{XY}$ is a dg map. It is injective as it is the restriction of an isomorphism to a subcategory of $\cat A$. If $v:FX\to FY$ is a cycle, let $u=(u_g:X\to gY)$ be the morphisms in $\cat A$ with $v=\sum\gamma^g\circ F(u_g)$. Then $$0=dv=\sum(-1)^{|\gamma^g|}\gamma^g\circ F(du_g)$$ and so $du=0$ by injectivity. Hence, $F$ restricts to a Galois covering $Z_*\cat A\to Z_*\cat B$.

For two objects $X$ and $Y$ of $\cat A$ ,we claim that the inverse image of $B_*\cat B(FX,FY)$ is $(B_*\cat A)^G(X,Y)$, and so the Lemma will follow from Lemma \ref{lem:CoveringIdeals}. Clearly a boundary $X\to gY$ maps to a boundary, since $\gamma^g_Y$ is a cycle. If $u:X\to gY$ maps to a boundary, then $\gamma^g_Y\circ F(u)=ds$. Let $s'$ be the unique element of $\cat A(X,gY)$ with $\gamma^g_Y\circ F(s')=s$. Then $$\gamma^g_Y\circ F(ds')=(-1)^{|\gamma^g_Y|}d(\gamma^g_F\circ F(s))=(-1)^{|\gamma^g_Y|}ds$$ and hence $u=\pm ds'$ by injectivity.
\end{proof}

\begin{proposition}\label{prop:CoveringHomology}
Given a section $S$ of $F$, the section category $\cat A^S$ is a dg category. Moreover, $S$ defines a section of $H_*F$, and the section category $(H_*\cat A)^S$ is isomorphic to $H_*(\cat A^S)$. \qed
\end{proposition}


\section{Derived categories via differential graded categories}\label{sec:DerivedCategories}

The shift functor of the bounded derived category $\Db(\cat A)$ of an Abelian category endows it with a $\Z$-action. Moreover this action is free since we are working with bounded complexes. The goal of this section is to study the orbit category $\T^\Z$ of a general triangulated category $\T$ whose translation functor $S$ acts freely, and in particular the case where $\T=\Db(\cat A)$ is the bounded derived category of a (skeletally small) Abelian category $\cat A$. Using the covering theory developed in the previous section and homotopy perturbation theory, we endow $\Db(\cat A)^\Z$ with a minimal $A_\infty$-structure having desirable compatibility properties with respect to the stabilization morphisms $SX\to X$. Finally, we prove Theorem D.

\subsection{Graded categories via actions} 

When $G=\Z$, the orbit category $\T^\Z$ is a graded category. A category with free $\Z$-action is just a category with a distinguished automorphism $S$ given by the action of $1\in\Z$. A category $\T$ with fixed automorphism $S$ will be called a \emph{$\Z$-category}. In order to minimize confusion between composition in $\T$ and composition in $\T^\Z$, we denote $g\bullet f=S^p(g)\circ f$ for $f\in\T^\Z_p(X,Y)$.

If $F:\cat T\to\cat D$ is $\Z$-stable, any stabilization $\gamma$ is uniquely determined by the natural isomorphism $\gamma^1:F\circ S\xrightarrow{\sim} F$; conversely, any natural isomorphism $\gamma:F\circ S\xrightarrow{\sim} F$ determines a $\Z$-stabilization of $F$ by setting $$\gamma^n_X=\gamma_{S^{n-1}X}\circ\gamma_{S^{n-2}X}\circ\cdots\circ\gamma_X.$$

The canonical $\Z$-stabilization $\sigma$ of $\pi:\T\to\T^\Z$ is determined by the degree $+1$ morphisms $\sigma_X:SX\to X$ of $\T^\Z$, so $\sigma:\pi\circ S\xrightarrow{\sim} \pi$ is a degree $+1$ natural transformation. By slight abuse of terminology, such stabilization will be said to have \emph{degree} $+1$.

\begin{definition}
The ``canonical'' stabilization $\sigma$ is only canonical up to twisting by a character $\chi:\Z\to\k^\times$. In particular, twisting by $\chi(n)=(-1)^n$ gives a stabilization $\bar{\sigma}=\sigma^\chi$ with $\bar{\sigma}_X=-\sigma_X$.  We refer to $\bar{\sigma}$ as the \emph{anticanonical} stabilization.
\end{definition}

If $\cat D$ is a graded category, and $F:\T\to\cat D$ is a $\Z$-stable graded functor where we view $\T$ as a graded category concentrated in degree 0, the induced functor $\bar{F}:\T^\Z\to\cat D$ is graded as well provided that the stabilization $\gamma:F\circ S\to F$ is homogeneous with degree $+1$.

A functor $F:(\T,S)\to(\T',S')$ between $\Z$-categories is a \emph{$\Z$-functor} if it strictly commutes with the automorphisms $S$ and $S'$. The following is an immediate corollary of the above lemma.

\begin{corollary}
A $\Z$-functor $F:(\T,S)\to(\T',S')$ induces a graded functor $F^\Z:\T^\Z\to(\T')^\Z$ so that the diagram
$$\xymatrix{\T\ar[r]^F\ar[d]^\pi & \T'\ar[d]^{\pi'} \\ \T^\Z\ar[r]^{F^\Z} & (\T')^\Z}$$ of categories and functors commutes. \qed
\end{corollary}

The action of the automorphism $S$ extends to a graded automorphism $S^\Z$ of $\T^\Z$ with the same action on objects as $S$, and action on morphisms determined by the maps
\begin{gather*}
S^\Z_{XY}:\T^\Z_n(X,Y)\to\T^\Z_n(S^\Z X, S^\Z Y) \\
f\mapsto (-1)^nS_{X,S^nY}(f)
\end{gather*}
where $S_{X,S^nY}:\T(X,S^nY)\to\T(SX, S^{n+1}Y)$ is the canonical map.

We now record some elementary properties about morphisms in the associated graded category $\T^\Z$. 

\begin{lemma} 
\begin{enumerate}
\item The functor $S^\Z:\T^\Z\to\T^\Z$ is induced by the $\Z$-functor $S:\T\to\T$ provided that the functor $\pi:\T\to\T^\Z$ is stabilized by the anticanonical stabilization $\bar{\sigma}$.
\item If $f$ is homogenous of degree $n$ then $f\bullet\sigma_X=(-1)^n\sigma_Y\bullet S^\Z(f)$.
\end{enumerate}
\end{lemma}

\begin{proof}
By universality of the associated graded category, $\bar{S}(\sigma_X)=-\sigma_X$. Given a degree $n$ morphism $f$, the composition $\sigma^{-n}_Y\bullet f$ is a degree 0 morphism, and so lies in $\T$. But then 
$$\bar{S}(f)=\bar{S}(\sigma^n_Y)\bullet S(\sigma^{-n}_Y\bullet f)=(-1)^n\sigma^n_Y\bullet S(f)=(-1)^nS(f)$$ as a degree $n$ morphism $X\to Y$ in $\T^\Z$. Hence $\bar{S}(f)=S^\Z(f)$, completing the proof. 
The second statement follows immediately from the definitions.
\end{proof}

\subsection{DG structures on orbit categories}

Let $\cat A$ be a graded category such that the degree 0 subcategory $\cat A_0$ is a $\Z$-category with automorphism $S$.

\begin{definition}\label{def:GeneratedOverZero}
The category $\cat A$ is said to be \emph{generated over degree 0} if there is a degree $+1$ natural isomorphism $s:S\xrightarrow{\sim}\Id_0$ where $\Id_0$ is the restriction of the identity functor $\Id_\cat A$ to the degree 0 subcategory $\cat A_0$. Here by a degree $+1$ natural transformation $s:S\to\Id_0$ we mean a natural transformation with degree $+1$ components, i.e., $s_X\in\cat A_1(SX,X)$. \\
If $\cat A$ is a dg category, it is said to be \emph{dg generated over degree 0} if the components $s_X:SX\to X$ of the natural isomorphism $s$ are cycles in $\cat A(SX,X)$.
\end{definition}

\begin{remark}
Note that an isomorphism $s$ in a dg category is a cycle if and only if the same is true for its inverse. Indeed, for such $s$ one has $d(s)=-(-1)^{|s|}s\circ d(s^{-1})\circ s$.
\end{remark}

\begin{example}
If $\T$ is a $\Z$-category, the orbit category $\T^\Z$ is generated over degree 0. 
\end{example}

Generation over degree 0 imposes a great deal of structure on the graded category $\cat A$. Roughly, generation over degree 0 means that most of $\cat A$ can be recaptured from its degree 0 subcategory.

\begin{lemma}\label{lem:ShiftAndSuspension}
If $\cat A$ is generated over degree 0, then the automorphism $S$ of $\cat A_0$ extends to a graded automorphism $S^\Z$ of $\cat A$ with 
$$s_Y\circ S^\Z(f)=(-1)^nf\circ s_X$$ for $f:X\to Y$ homogeneous of degree $n$. Moreover, if $\cat A$ is a dg category dg generated over degree 0, the automorphism $S^\Z$ is a dg automorphism.
\end{lemma}

\begin{proof}
We first remark that any degree 0 morphism $f$ necessarily satisfies $S(f)=s_Y^{-1}\circ f\circ s_X$ by naturality of $s$. For a homogeneous degree $n$ morphism $f:X\to Y$, define $S^\Z(f)=(-1)^ns^{-1}_Y\circ f\circ s_X$. The association $S^\Z$ is clearly functorial, and satisfies the condition of the Lemma.

Suppose now that $\cat A$ is dg, and $f:X\to Y$ is homogeneous of degree $n$. Then 
$$d_{XY}(S^\Z f)=(-1)^nd_{XY}(s_Y^{-1}\circ f\circ s_X)=(-1)^{n-1}s_Y^{-1}\circ d_{XY}(f)\circ s_X=S^\Z(d_{XY}(f))$$ since $s_Y^{-1}$ and $s_X$ are cycles and the degree of $s_Y^{-1}$ is $-1$. Hence $S^\Z$ is a dg functor.
\end{proof}

\begin{lemma}\label{lem:GeneratedOverZero}
If $\cat A$ is a (differential) graded category (dg) generated over degree 0, then the maps 
\begin{align*}
\begin{array}{c@{\mskip\thickmuskip}l}
s_*:\cat A(X,Y)\to\cat A(X,S^\Z Y)[1] \\
f\mapsto s^{-1}_Y\circ f
\end{array}
\qquad \qquad 
\begin{array}{c@{\mskip\thickmuskip}l}
s^*:\cat A(X,Y)\to\cat A(S^\Z X,Y)[-1] \\
f\mapsto (-1)^{|f|}f\circ s_X
\end{array}
\end{align*}

are degree 0 (differential) graded isomorphisms. Moreover, the diagram $$\xymatrix@R=36pt@C=36pt{\cat A(X,Y)\ar[r]^{s_*}\ar[d]^{s^*}\ar[dr]^{S^\Z_{XY}} & \cat (X,S^\Z Y)[1]\ar[d]^{s^*[1]} \\
\cat A(S^\Z X,Y)[-1]\ar[r]_{s_*[-1]} & \cat A(S^\Z X,S^\Z Y) }$$
of (differential) graded spaces and (differential) graded homomorphisms commutes.
\end{lemma}

\begin{proof} 
It is clear that $s_*$ and $s^*$ are graded isomorphisms; let us show that they are compatible with differentials. Indeed, for a homogeneous degree $n$ morphism $f:X\to Y$ one calculates

\begin{gather*}
d_{X,SY}[1](s_*(f))=-d_{X,SY}(s_Y^{-1}\circ f)=s_Y\circ d_{XY}(f)= s_*(d_{XY}(f)) \\
d_{SX,Y}[-1](s^*(f))=-(-1)^nd_{SX,Y}(f\circ s_X)=(-1)^{n+1}d_{XY}(f)\circ s_X=s^*(d_{XY}(f))
\end{gather*}
since $d[\pm 1]=-d$.

Since $S^\Z(f)=(-1)^ns_X\circ f\circ s_Y^{-1}$, in the following diagram 

$$\xymatrix{\cat A_n(X,Y)\ar[r]\ar[d]\ar[dr] & \cat A_{n-1}(X,S^\Z Y)\ar[d] \\ \cat A_{n+1}(S^\Z X,Y)\ar[r] & \cat A_n(S^\Z X,S^\Z Y)}$$

the-top right triangle commutes, while the bottom-left triangle anti-commutes. After shifting, the map
$s^*[-1]:\cat A(S^\Z X,Y)[-1]\to\cat A(S^\Z X,S^\Z Y)$ sends a degree $n+1$ morphism $g:S^\Z X\to Y$ of $\T^\Z$ to $(-1)^ng\circ s_X$ since such $g$ has degree $n$ in the shifted morphism space $\cat A(S^\Z X,Y)[-1]$. Therefore the diagram commutes after shifting gradings, completing the proof.
\end{proof}

\begin{corollary}
The morphisms $s_*$ and $s^*$ induce a pair of mutually inverse natural transformations of respective degrees $-1$ and $1$ between $S^\Z$ and the identity functor. \qed
\end{corollary}

The above corollary implies that the inclusion functor $\cat A_0\to\cat A$ is $\Z$-stable, and so induces a graded functor $(\cat A_0)^\Z\to\cat A$.

\begin{proposition}
If $\cat A$ is generated over degree 0, the induced functor 
\begin{equation}\label{eqn:GradedOverZero}
(\cat A_0)^\Z\xrightarrow{\sim}\cat A
\end{equation} is an equivalence of graded categories.
\end{proposition}

\begin{proof}
Recall that the functor $(\cat A_0)^\Z\to\cat A$ is given by the degree 0 maps
\begin{gather*}
(\cat A_0)^\Z(X,Y)\to\cat A(X,Y) \\
f\mapsto (-1)^ns^n_Y\circ f
\end{gather*}
for $f:X\to S^nY$ in $\cat A_0$. By repeated application of Lemma \ref{lem:GeneratedOverZero}, this map is an isomorphism, hence the functor is fully faithful. Since both $(\cat A_0)^\Z$ and $\cat A$ have the same objects as $\cat A_0$, it is an equivalence.
\end{proof}

When $\cat A$ is a dg category, the equivalence \eqref{eqn:GradedOverZero} induces a differential on the associated graded category which we now explicitly describe. The degree 1 isomorphisms $s_X:SX\to X$ induce isomorphisms 
$$s^n_*:\cat A_n(X,Y)\xrightarrow{\sim}\cat A_0(X,S^nY)[n]$$ for each $n\in\Z$. The differential $\partial$ of $(\cat A_0)^\Z$ is given by 
\begin{gather*}
\partial_{XY}:\cat A_0(X,S^nY)\to\cat A_0(X,S^{n-1}Y) \\
f\mapsto(-1)^ns_{S^{n-1}Y}\circ d_{X,S^nY}(f)
\end{gather*}
on the space $(\cat A_0)^\Z(X,Y)$.

\begin{proposition}
The maps $\partial_{XY}$ endow $(\cat A_0)^\Z$ with the structure of a dg category, and the equivalence $(\cat A_0)^\Z\xrightarrow{\sim}\cat A$ is a dg functor.
\end{proposition}

\begin{proof}
We first verify that the maps $\partial_{XY}$ are differentials. Indeed, 
\begin{multline*}
\partial^2_{XY}=\partial_{XY}(s_{S^{n-1}Y}\circ d_{X,S^nY}f)\\
=s_{S^{n-2}Y}\circ d_{X,S^{n-1}Y}(s_{S^{n-1}Y}\circ d_{X,S^nY}f)=-s^2_{S^{n-2}Y}d^2_{X,S^nY}f=0
\end{multline*}
since $s_{S^{n-1}Y}$ is a cycle and $d^2=0$. 

Next we verify the Leibniz law for $\partial$. Let $f:X\to S^pY$ and $g:Y\to S^qZ$ be morphisms in $\cat A_0$, and set $n=p+q$. Then 
\begin{multline*}
\partial_{XZ}(g\bullet f) = (-1)^ns_{S^{n-1}Z}\circ d_{X,S^nZ}(S^p(g)\circ f) \\
=(-1)^ns_{S^{n-1}Z}\circ d_{S^pY,S^nZ}(S^p(g))\circ f+(-1)^ns_{S^{n-1}Z}\circ S^p(g)\circ d_{X,S^pY}(f) \\
=(-1)^ns_{S^{n-1}Z}\circ S^p(d_{Y,S^qZ}(g))\circ f+(-1)^ns_{S^{n-1}Z}\circ S^p(g)\circ d_{X,S^pY}(f)
\end{multline*}
since $S$ is a dg functor.
On the other hand, 
\begin{multline*}
\partial_{YZ}(g)\bullet f+(-1)^qg\bullet\partial_{XY}(f) \\
=(-1)^q(s_{S^{q-1}Z}\circ d_{Y,S^qZ}(g))\bullet f+(-1)^ng\bullet(s_{S^{p-1}Y}\circ d_{X,S^pY}(f)) \\
=(-1)^qS^p(s_{S^{q-1}Z}\circ d_{Y,S^qZ}(g))\circ f+(-1)^nS^{p-1}(g)\circ s_{S^{p-1}Y}\circ d_{X,S^pY}(f).
\end{multline*}
The first summand equals $(-1)^ns_{S^{n-1}Z}\circ d_{S^pY,S^nZ}(S^p(g))$, and the second summand equals $(-1)^ns_{S^{n-1}Z}\circ S^p(g)\circ d_{X,S^pY}(f)$ by Lemma \ref{lem:ShiftAndSuspension}. Hence the Leibniz law holds and so $(\cat A_0)^\Z$ is a dg category.

To see that the functor $(\cat A_0)^\Z\to\cat A$ is a dg functor, we need to show that the diagram
$$\xymatrix{\cat A_0(X,S^nY)\ar[r]^{s^n_*}\ar[d]^{\partial_{XY}}& \cat A_n(X,Y)\ar[d]^{d_{XY}} \\ \cat A_0(X,S^{n-1}Y)\ar[r]^{s^{n-1}_*} & \cat A_{n-1}(X,Y)}$$
commutes. One calculates that
\begin{gather*}
(s^{n-1}_*\circ\partial_{XY})(f)=(-1)^ns^{1-n}_Y\circ s_{S^{n-1}Y}\circ d_{X,S^nY}(f)\\
=(-1)^ns^n_Y\circ d_{X,S^nY}(f)=d_{XY}(s^n_Y\circ f)=(d_{XY}\circ s^n_*)(f)
\end{gather*}
so the equivalence is a dg equivalence. This completes the proof of the proposition.
\end{proof}

\begin{proposition}\label{prop:HomologyOverZero}
For $\cat A$ as in the previous proposition, the homology $H_*\cat A$ is generated over degree 0.
\end{proposition}

\begin{proof}
The previous proposition implies that there is an isomorphism of chain complexes $$(\cat A_0)^\Z(X,S^nY)\xrightarrow{\sim}\cat A(X,S^nY).$$ Inspection at degree 0 gives a commutative diagram
$$\xymatrix{\cat A_1(X, S^nY)\ar[r]\ar[d] & \cat A_0(X,S^nY)\ar[r]\ar[d] & \cat A_{-1}(X,S^nY)\ar[d] \\ 
\cat A_0(X, S^{n+1}Y)\ar[r] & \cat A_0(X,S^nY)\ar[r] & \cat A_0(X,S^{n-1}Y)}$$
where the vertical maps are isomorphisms and the rows are chain complexes. Accordingly, $(H_0\cat A)(X,S^nY)$ is isomorphic to the homology of the bottom row, which is evidently $H_n(A_0)^\Z(X,Y)$. Hence there is a chain of graded equivalences $$(H_0\cat A)^\Z\xrightarrow{\sim} H_*((\cat A_0)^\Z)\xrightarrow{\sim} H_*\cat A$$ proving that $H_*\cat A$ is generated over degree 0.
\end{proof}

\subsection{Skeleta} 

Recall that a category is \emph{skeletal} if all isomorphisms are automorphisms. A \emph{skeleton} of a category $\cat C$ is a skeletal full subcategory $\sk(\cat C)$ such that the inclusion $\sk(\cat C)\hookrightarrow\cat C$ is essentially surjective, and $\cat C$ is \emph{skeletally small} if it admits a small skeleton.

\begin{definition}
If $\T$ is a $\Z$-category, a skeleton $\sk(\T)$ is \emph{$\Z$-stable} if it is closed under the action of the automorphism $S$.
\end{definition}

\begin{lemma}
Suppose $\T$ is a directed skeletally small $\Z$-category. Then $\T$ admits a $\Z$-stable small skeleton $\sk(\T)$. 
\end{lemma}

\begin{proof}
Choose a skeleton $\sk(\T^\Z)$ of the orbit category. Let $\pi^{-1}\sk(\T^\Z)$ be the full subcategory of $\T$ whose objects $X$ map to objects in $\sk(\T^\Z)$ under the functor $\pi$, and define $\sk(\T)$ to be the full subcategory of $\T$ such that $S^nX$ is in $\pi^{-1}\sk(\T^\Z)$ for some $n\in\Z$. Note the category $\sk(\T)$ is manifestly $\Z$-stable; we claim it is in fact a $\Z$-stable skeleton of $\T$.

Suppose that $f:X\to Y$ is an isomorphism in $\sk(\T)$. Let $p,q\in\Z$ so that $S^pX$ and $S^qY$ are in $\pi^{-1}\sk(\T^\Z)$. Then $\pi S^pX$ and $\pi S^qY$ are in $\sk(\T^\Z)$, and $$\sigma^{-q}_Y\bullet f\bullet\sigma^p_X:\pi S^pX\to\pi S^qY$$ is an isomorphism in $\T^\Z$. Hence it is an automorphism, and so $S^pX=S^qY$.\break Then $\T(X,S^{q-p}Y)$ and $\T(S^{q-p}Y,X)$ are both non-zero, as they contain the shift of $id_{S^pX}=id_{S^qY}$. But $\T(X,Y)$ and $\T(Y,X)$ are non-zero, containing $f$ and its inverse respectively. Since $\T$ is directed, we must have $q-p=0$, so $X=Y$ and $f$ is an automorphism. 
\end{proof}

The notion of a skeleton is too rigid for the study of $\Z$-categories. Upon passing to a skeleton of the orbit category $\T^\Z$ one loses the stabilization isomorphisms $\sigma_X:SX\to X$. To retain this structure we introduce the following notion.

\begin{definition}
Let $\gamma$ be a distinguished collection of isomorphisms in $\cat C$. A \emph{$\gamma$-pseudo-skeleton} of $\cat C$ is a full subcategory $\psk_\gamma(\cat C)$ (or simply $\psk(\cat C)$ if $\gamma$ is understood) of $\cat C$ in which 
\begin{enumerate}
\item all isomorphisms are compositions of automorphisms and isomorphisms in $\gamma$ and 
\item the inclusion $\psk_\gamma(\cat C)\hookrightarrow\cat C$ is essentially surjective.
\end{enumerate}
\end{definition}

\begin{examples}
\begin{enumerate}
\item If $\gamma$ is the collection of automorphisms in $\cat C$, then a $\gamma$-pseudo-skeleton is simply a skeleton. 
\item If $\gamma$ is the collection of all isomorphisms, then any category equivalent to $\cat C$ is a $\gamma$-pseudo-skeleton.
\end{enumerate}
\end{examples}

\begin{lemma}\label{lem:PseudoSkeletons}
Let $\T$ be a $\Z$-category with $\Z$-stable skeleton $\sk(\T)$ and let $\gamma$ be a stabilization of the functor $\pi:\T\to\T^\Z$. Then $$\psk(\T^\Z)=\pi(\sk(\T))$$ is a $\gamma$-pseudo-skeleton for $\T^\Z$. In particular, if $\T$ is directed, it admits a $\gamma$-pseudo-skeleton.
\end{lemma}

\begin{proof}
Suppose that $f:X\to Y$ is an isomorphism in $\T^\Z$, with $X$ and $Y$ in $\sk(\T)$. Note that $f=\sigma^n_Y\bullet\bar{f}$ for some $\bar{f}:X\to S^nY$ in $\T$, namely $\bar{f}=\sigma^{-n}_Y\bullet f$. Since $f$ is invertible in $\T^\Z$, $\bar{f}$ is invertible in $\T$ with inverse $\bar{f}^{-1}=f^{-1}\bullet\sigma_Y^n$, and so is an automorphism, i.e., $S^nY=X$. Hence $f$ is the composition of an automorphism $\bar{f}$ of $X$ with $\sigma^n_Y$, proving the Lemma.
\end{proof}


\subsection{Minimal model of the derived category} 

We now specialize to the case $\T=\Db(\cat A)$ of an Abelian category $\cat A$.

\begin{lemma}
If $\cat A$ is an abelian (ordinary, graded, or dg) category, then the category $\Cgr(\cat A)$ (resp. $\Cdg(\cat A)$) is (resp. dg) generated over degree 0.
\end{lemma}

\begin{proof}
Note that it suffices to prove the claim for $\Cgr(\cat A)$, since $\Cdg(\cat A)$ is a full subcategory of $\Cgr(\cat A)$ closed under suspension. For a graded space $X$, the identity map of the underlying ungraded space gives a canonical degree $1$ isomorphism $s_X:X[1]\to X$ between graded spaces, and $s_Y\circ f[1]=(-1)^nf\circ s_X$ for $f:X\to Y$ a degree $n$ graded map. In particular, if $f$ has degree 0, then $s_Y\circ f[1]=f\circ s_X$, so the isomorphisms $s_X:X[1]\to X$ give a natural equivalence $s:[1]\xrightarrow{\sim}\Id_0$. Thus $\Cgr(\cat A)$ is generated over degree 0. The morphisms $s_X$ are cycles in $\Cdg(\cat A)$ completing the proof.
\end{proof}

The subcategory $\Pdg(\cat A)$ of projective complexes in $\Cdg(\cat A)$ is closed under suspension, and hence is also dg generated over degree 0. Since the bounded derived category of $\cat A$ is equivalent to $H_0\Pdg(\cat A)$, the above Lemma combined with Proposition \ref{prop:HomologyOverZero} gives the following:

\begin{corollary}\label{cor:AinftyEnvelope}
There is a graded equivalence of categories $$H_*\Pdg(\cat A)\xrightarrow{\sim}\Db(\cat A)^\Z. $$ \qed
\end{corollary}

Thus by Kad\-ei\v{s}\-vili's Theorem, $\Db(\cat A)^\Z$ admits the structure of a minimal $A_\infty$-cat\-e\-gory. The orbit category $\Db(\cat A)^\Z$ can be thought of as an \emph{$A_\infty$-envelope} of the triangulated category $\Db(\cat A)$, in the sense that it contains $\Db(\cat A)$ as its degree zero subcategory, and the composition in $\Db(\cat A)$ is given by the $A_\infty$-composition $\mu_2$ of $\Db(\cat A)^\Z$.

We wish to construct an explicit model of this $A_\infty$-structure by applying the Homotopy Transfer Theorem, but this requires us to replace the category $\Pdg(\cat A)$ by a small dg category. In order for the resulting $A_\infty$-structure to be compatible with the stabilization $\sigma:[1]\xrightarrow{\sim}\Id$, we use $\sigma$-pseudo-skeleta.

\subsection{Invariant splittings} 

Suppose $\cat A$ is a small dg category dg generated over degree 0, and $\gamma$ a stabilization of the inclusion $\cat A_0\to\cat A$ by cycles.

\begin{definition}\label{def:GammaInvariant} 
A collection of homotopy retractions $(j_{XY},q_{XY},\phi_{XY})$ of the morphism spaces $\cat A(X,Y)$ is \emph{$\gamma$-invariant} if in the following diagram of homotopy retractions
$$\xymatrix{
\cat A(X,SY)\ar[r]^{(\gamma_Y)_*}\ar@/_/[d]_{q_{X,SY}}\ar@(ul,ur)^{\phi_{X,SY}} & \cat A(X,Y)\ar@/_/[d]_{q_{XY}}\ar[r]^{\gamma_X^*}\ar@(ul,ur)^{\phi_{XY}} & \cat A(SX,Y)\ar@/_/[d]_{q_{SX,Y}}\ar@(ul,ur)^{\phi_{SX,Y}} \\
H_*\cat A(X,SY)\ar[r]^{[\gamma_Y]_*}\ar@/_/[u]_{j_{X,SY}} & H_*\cat A(X,Y)\ar[r]^{[\gamma_X]^*}\ar@/_/[u]_{j_{XY}} & H_*\cat A(SX,Y)\ar@/_/[u]_{j_{SX,Y}}}$$

the equalities

\begin{align}\label{eqn:InvariantRetractions}
\begin{array}{c@{\mskip\thickmuskip}l}
q_{XY}\circ(\gamma_Y)_*=[\gamma_Y]^*\circ q_{X,SY} \\
j_{XY}\circ[\gamma_Y]_*=\gamma_Y^*\circ j_{X,SY} \\
\phi_{XY}\circ(\gamma_Y)_*=(\gamma_Y)_*\circ \phi_{X,SY}
\end{array}
\qquad \qquad
\begin{array}{c@{\mskip\thickmuskip}l}
q_{SX,Y}\circ(\gamma_X)_*=[\gamma_X]^*\circ q_{XY} \\
j_{SX,Y}\circ[\gamma_X]_*=\gamma_X^*\circ j_{XY} \\
\phi_{SX,Y}\circ\gamma_X^*=\gamma_X^*\circ \phi_{XY}
\end{array}
\end{align} 

hold.
\end{definition}

If the degree zero subcategory of $\cat A$ is directed, then by Lemma \ref{lem:PseudoSkeletons} it admits a $\sigma$-pseudo-skeleton for the isomorphisms given by the stabilization $\sigma$ of the suspension functor.

\begin{lemma}\label{lem:InvariantSplittings}
If $\cat A_0$ is directed and skeletally small, then the dg category $\psk(\cat A)$ admits $\sigma$-invariant homotopy retractions.
\end{lemma}

\begin{proof}
Since $\cat A_0$ is directed, Lemma \ref{lem:PseudoSkeletons} gives a $\sigma$-pseudo-skeleton for $\cat A$. By construction, the objects of $\psk\cat A$ are of the form $S^nX$ for some $X$ with $\pi X$ an object of a fixed small skeleton $\sk(\cat A)$. Choose homotopy retractions $(j_{XY},q_{XY},\phi_{XY})$ for such $X$ and $Y$. For arbitrary $S^pX$ and $S^qY$, the equations \eqref{eqn:InvariantRetractions} uniquely determine $\sigma$-invariant homotopy retractions of each $\cat A(X,Y)$.
\end{proof}

\subsection{Suspension invariance}

Let $(\mu_n)_{n\geq 2}$ be the minimal $A_\infty$-structure on the pseudo-skeleton of $H_*\cat A$ given by applying the Homotopy Transfer Theorem to the splittings of Lemma \ref{lem:InvariantSplittings}.

\begin{proposition}\label{prop:SuspensionInvariance}
The $A_\infty$-compositions $\mu_n$ are $[\sigma]$-invariant in the sense that

\begin{gather*} 
[\sigma]_*\circ\mu_n = \mu_n\circ\left([\sigma]_*\otimes id^{\otimes(n-1)}\right)\\
[\sigma]^*\circ\mu_n  =  \mu_n\circ\left(id^{\otimes(n-1)}\otimes [\sigma]^*\right)
\end{gather*}
and
\begin{gather*} 
\mu_n\circ\left(id^{\otimes(k-1)}\otimes([\sigma]^*\otimes id)\otimes id^{\otimes(n-k-1)}\right) \\
= \mu_n\circ\left(id^{\otimes(k-1)}\otimes(id\otimes[\sigma]_*)\otimes id^{\otimes(n-k-1)}\right) \end{gather*}

for $0<k<n$.
\end{proposition}

\begin{proof}
Recall that $\mu_n$ is given as a sum of maps $\mu_T$ over PBR $n$-trees $T$ so it suffices to show that the maps $\mu_T$ satisfy the formulae of the proposition. The idea of the proof is simple: The map $\mu_T$ is made by replacing the edges of $T$ by morphisms which are all $\sigma$-equivariant. So, a $[\sigma]$ term can be freely pushed through the tree from one leaf to another.

To prove the proposition formally, let $\nu_T$ be the map constructed in Construction \ref{constr:NuMaps} so that $\mu_T=q\circ\nu_T\circ j^{\otimes n}$. Since the splittings $(j,q,\phi)$ are $\sigma$-invariant, it suffices to show that the $\nu_T$ satisfy formulae analogous to those of the proposition. We proceed by induction on $T$.

Note that for the unique PBR 2-tree $Y$, the map $\nu_Y$ is just ordinary composition in $\cat A$, which manifestly satisfies the above formulae.

For arbitrary $T$, we are done by induction except for the following cases:
\begin{enumerate}
\item If the left subtree of $T$ is a leaf we still need to show $$[\sigma]_*\circ\nu_T=\nu_T\circ\left([\sigma]_*\otimes id^{\otimes(n-1)}\right),$$
\item if $0<k<n$ and the $k$-th leaf is on the left subtree of $T$ and the $(k+1)$-st leaf is on the right subtree of $T$, we still need to show 
\begin{gather*}
\nu_T\circ\left(id^{\otimes(k-1)}\otimes([\sigma]^*\otimes id)\otimes id^{\otimes(n-k-1)}\right) \\
=\nu_T\circ\left(id^{\otimes(k-1)}\otimes(id\otimes[\sigma]_*)\otimes id^{\otimes(n-k-1)}\right),
\end{gather*}
\item and if the right subtree of $T$ is a leaf we still need to show $$[\sigma]^*\circ\nu_T=\nu_T\circ\left(id^{\otimes(n-1)}\otimes [\sigma]^*\right).$$
\end{enumerate}

Let us prove the second case; the other two are similar. Recall that the left and right subtrees $T^-$ and $T^+$ of $T$ are the PBR trees obtained from $T$ by deleting a small neighborhood around the unique vertex adjacent to the root edge of $T$ (\emph{cf.} Section \ref{sec:HTT}). Since 
$\nu_T=\mu\circ(\phi\circ\nu_{T^-})\otimes(\phi\circ\nu_{T^+})$, the left-hand side of the equation is

\begin{gather*}
\nu_T\circ\left(id^{\otimes(k-1)}([\sigma]^*\otimes id)\otimes id^{\otimes(n-k-1)}\right) \\
=\mu\circ(\phi\circ\nu_{T^-})\otimes(\phi\circ\nu_{T^+})\circ\left(id^{\otimes(k-1)}\otimes [\sigma]^*\otimes id^{\otimes(n-k)}\right) \\
=\mu\circ\left(\phi\circ\nu_{T^-}\circ(id^{\otimes(k-1)}\otimes [\sigma]^*)\otimes(\phi\circ\nu_{T^+}\circ id^{\otimes(n-k)})\right) \\
=\mu\circ\left(\phi\circ\nu_{T^-}\circ(id^{\otimes(k-1)})\otimes([\sigma]_*\circ\phi\circ\nu_{T^+}\circ id^{\otimes(n-k)})\right)
\end{gather*}
by induction. Similarly, the right-hand side is equal to $$\mu\circ\left(\phi\circ\nu_{T^-}\circ(id^{\otimes(k-1)})\otimes(\phi\circ[\sigma]_*\circ\nu_{T^+}\circ id^{\otimes(n-k)})\right),$$
which equals the left-hand side by $\sigma$-equivariance of the homotopy $\phi$. This completes the proof of the proposition.
\end{proof}

\subsection{Triangles}

The homology $H_*\Pdg(\cat A)$ is equivalent to $\Db(\cat A)^\Z$ by Corollary \ref{cor:AinftyEnvelope}, so $\mu_3$ maps triples of morphisms in $\Db(\cat A)$ into $\Db(\cat A)^\Z_1$. A slight rephrasing of Lemma 3.7 from \cite{Seidel} to our situation gives the following.

\begin{proposition}[\cite{Seidel} Lemma 3.7]\label{prop:Triangles}
If $\xymatrix@C=12pt{X\ar[r]^f & Y\ar[r]^g & Z\ar[r]^h & X[1]}$ is a non-split triangle in $\Db(\cat A)$ then $$\mu_3(h,g,f)=s_X$$ where $s_X:X\to X[1]$ is the canonical degree $1$ map. \qed
\end{proposition}

\subsection{Vanishing of higher compositions}

Recall that an abelian category $\cat A$ is \emph{hereditary} if $\Ext^n_\cat A(X,Y)$ vanishes for $n>1$ and any two objects $X$ and $Y$. We say an abelian category is \emph{directed} if its bounded derived category is directed with respect to the shift functor.

For hereditary $\cat A$, every indecomposable object of $\Db(\cat A)$ is of the form $X[d]$ for some $d\in\Z$ and indecomposable object $X$ of $\cat A$ (\emph{cf. e.g.,}  \cite{Happel2}). Thus, if $X[p]$ and $Y[q]$ are objects of $\Db(\cat A)$ with $X$ and $Y$ in $\cat A$, $$\Db(\cat A)(X[p],Y[q])=\Ext_{\cat A}^{q-p}(X,Y).$$ A hereditary category is directed, since the non-vanishing of both $\Ext_\cat A^n(X,Y)$ and $\Ext_{\cat A}^{-n}(Y,X)$ implies $n=0$.

\begin{proposition}\label{prop:HighlyTrivial}
If $\cat A$ is a hereditary category with enough projectives, then $\mu_n$ vanishes for $n>3$.
\end{proposition}

\begin{proof}Consider a sequence
$$\xymatrix{X_0\ar[r]^{f_1}& X_1\ar[r]^{f_2} &\cdots\ar[r]^{f_n} & X_n}$$ of non-zero composeable morphisms in $\Db(\cat A)^\Z$. By additivity of $\mu_n$, we may assume that all $X_k$ are indecomposable. Since $\cat A$ is directed, all the $f_k$ are of non-negative degree, say $d_k$. Then $X_0$ is concentrated in degree $d_0$, $X_1$ is concentrated in degree $d_0+d_1$, and in general $X_k$ is concentrated in degree $d_0+d_1+\cdots +d_k$.

Thus, $\mu_n(f_n,\dots, f_1):X_0\to X_n$ is a degree $d_1+\cdots +d_n+n-2$ morphism, and so determines a degree $0$ morphism $X_0\to X_n[-(d_1+\cdots +d_n)-n+2]$. But $X_0'=X_0[-d_0]$ and $X_n'=X_n[-(d_0+\cdots +d_n)]$ are in $\cat A$ , and so this morphism represents class in $$\Db(\cat A)(X_0', X_n'[-n+2])=\Ext^{n-2}_{\cat A}(X_0', X_n'),$$ which is $0$ whenever $n>3$.
\end{proof}

By Lemma \ref{lem:InvariantSplittings}, if $\cat A$ is hereditary the orbit category $\Db(\cat A)^\Z$ admits a $\sigma$-invariant $A_\infty$-structure. Putting Propositions \ref{prop:SuspensionInvariance}, \ref{prop:Triangles}, and \ref{prop:HighlyTrivial} together and specializing to the bounded derived category of a hereditary abelian category $\cat A$ with enough projectives we get the following. 

\begin{theorem}\label{thm:DerivedCategoryAinfty}
If $\cat A$ is a hereditary abelian category with enough projectives then the minimal $A_\infty$-structure on $\Db(\cat A)^\Z$ satisfies
\begin{enumerate}
\item the compositions $\mu_n$ are $s$-equivariant for the maps $s_X:X\to X[1]$,
\item if $X\xrightarrow{f}Y\xrightarrow{g}Z\xrightarrow{h} X[1]$ is a non-split triangle in $\Db(\cat A)$ then $\mu_3(h,g,f)=s_X$,
\item the higher compositions $\mu_n$ vanish for $n\neq2,3$. 
\end{enumerate}
\qed
\end{theorem}


\section{Acyclic Ginzburg algebras}\label{sec:Ginzburg}

The covering theory of Section \ref{sec:Covering} will allow us to compute the minimal model for the Ginzburg algebra $\Gamma_Q$ using the $A_\infty$-structure of $\Db(Q)^\Z$ as constructed in Theorem \ref{thm:DerivedCategoryAinfty}.

\subsection{The derived translation algebra}

In order to relate the $A_\infty$-structure of $\Gamma_Q$ to $\Db(Q)^\Z$, we will need the following algebra.

\begin{definition}
The \emph{derived translation algebra} of $Q$ is the algebra 
$$U_Q=\bigoplus_{n\geq 0}\Db(Q)^\Z(\k Q,\tau^{-n}\k Q)$$
where the product of $f:\k Q\to\tau^{-p}\k Q$ and $g:\k Q\to\tau^{-q}\k Q$ is given by $$f\cdot g=\tau^{-q}(f)\circ g.$$
\end{definition}

The derived translation algebra is naturally bigraded: in addition to grading by morphism degree, we define the \emph{weight} of a morphism $f:\k Q\to\tau^{-n}\k Q$ to be $n$.

The isomorphism $\k Q=\Db(Q)(\k Q,\k Q)$ makes $\k Q$ a subalgebra of $U_Q$. This endows the derived translation algebra with the structure of a graded right $\k Q$-module.

\begin{proposition}\label{prop:TotalDegree}
Let $U^{\tot}_Q$ denote the (singly) graded $\k Q$-module given by considering $U_Q$ as graded by total degree, and set $F=\tau^-[1]$. Then there is a graded isomorphism of $\k Q$-modules $$U^{\tot}_Q\cong\bigoplus_{n\geq 0}F^n(\k Q)$$ where the right-hand side is graded by the usual degree of chain complexes. Moreover, under this isomorphism the submodule $F^n(\k Q)$ maps to the weight $n$ component of $U^{\tot}_Q$.
\end{proposition}

\begin{proof}
Denote by $P_i$ the indecomposable projective $\k Q$-module corresponding to the vertex $i$. Since
$$\Db(Q)^\Z(\k Q,\tau^{-n}\k Q)=\bigoplus_{i\in Q_0}\Db(Q)^\Z(\k Q,\tau^{-n}P_i)$$ it suffices to compute the graded $\k Q$-module structure of $\Db(Q)^\Z(\k Q,\tau^{-n}P_i)$.

Since $\k Q$ is hereditary, there are integers $n_0$ and $d$ such that $\tau^{-n}P_i$ is quasi-isomorphic to $\tau^{-n_0}P_i[d]$ with $\tau^{-n_0}P_i$ concentrated in degree 0. (If $Q$ is not Dynkin, then $n_0=n$ and $d=0$.) Hence $\Db(Q)^\Z(\k Q,\tau^{-n}P_i)$ is concentrated in degree $-d$. Moreover there are isomorphisms
$$\Db(Q)^\Z_{-d}(\k Q,\tau^{-n}P_i)=\Db(Q)(\k Q,\tau^{-n}P_i[-d])
=\Db(Q)(\k Q,\tau^{-n_0}P_i)=\tau^{-n_0}P_i.$$ 

But the complex $\Db(Q)^\Z_{-d}(\k Q,\tau^{-n}P_i)$ lives in degree $-d$, so $\Db(Q)^\Z(\k Q,\tau^{-n}P_i)$ is isomorphic to $\tau^{-n_0}P_i[d]=\tau^{-n}P_i$. Therefore the summand $\Db(Q)^\Z(\k Q,\tau^{-n}P_i)$ lies in total degree $n-d$, and so $\Db(Q)^\Z(\k Q,\tau^{-n}P_i)=F^nP_i$ as graded modules.
\end{proof}

\begin{remark}
The proof of Proposition \ref{prop:TotalDegree} actually shows that $U_Q$ is a preprojective algebra for $\Db(Q)$ in the sense that it is a graded algebra containing $\k Q$ as a subalgebra, and it splits as a direct sum into all indecomposable \emph{preprojective} objects in the derived category $\Db(Q)$. If $Q$ is not Dynkin then every preprojective object of $\Db(Q)$ is in fact a preprojective module (\emph{i.e.}, concentrated in degree 0).
\end{remark}

By Corollary \ref{cor:AinftyEnvelope} the algebra $U_Q$ is the homology of the dg algebra $$\bigoplus_{n\ge0}\Pdg(Q)(kQ,\tau^{-n}kQ)$$ and hence is naturally a minimal $A_\infty$-algebra by Kadei\v{s}vili's Theorem. (Here we really mean the quasi-isomorphic algebra obtained by replacing each $\tau^{-n}kQ$ by a chosen projective resolution.) 

The identity morphisms $id_X:X\to X$ give rise to canonical degree 1 maps $\sigma_X:X[1]\to X$ in $\Pdg(Q)$ and hence give a degree 1 natural equivalence $\sigma:[1]\xrightarrow{\sim}\Id$. Denote by $s_X$ the image of the degree $-1$ map $\sigma^{-1}_X:X\to X[1]$ in $U_Q$.

\begin{proposition}\label{prop:U}
The $A_\infty$-algebra $U_Q$ has the following properties:
\begin{enumerate}
\item the compositions $\mu_n$ are $s_X$-invariant,
\item if $f$, $g$, and $h$ are homogeneous of degree 0 and form a non-split distinguished triangle in $\Db(Q)$ then $\mu_3(h,g,f)=s_X$,
\item $\mu_n=0$ for $n\neq 2,3$.
\end{enumerate}
\end{proposition}

\begin{proof}
Consider the full subcategory $p(Q)$ of $\Pdg(Q)$ whose objects are given by choosing a fixed projective resolution for each of the objects $\tau^{-n}P_i$ with $n\ge 0$. It is a small dg category.
Since the homology $H_*p(Q)$ is a full subcategory of $\Db(Q)^\Z$, its $A_\infty$-structure satisfies the properties of the proposition by Theorem \ref{thm:DerivedCategoryAinfty}. Clearly the associated algebra $\k[p(Q)]$ is dg isomorphic to $\bigoplus_{n\ge 0}\Pdg(Q)(kQ,\tau^{-n}kQ)$ and hence its homology is $A_\infty$-isomorphic to $U_Q$. The proposition follows by applying Lemma \ref{lem:smallAinfty}.
\end{proof}

\subsection{Acyclic Ginzburg Algebras}\label{sec:GinzburgAlgebra}
Recall that the Ginzburg algebra can be endowed with a bigrading such that the arrows $\alpha$, $\alpha^*$ and $t_i$ have bidegrees $(0,0)$, $(1,0)$ and $(1,1)$ respectively. We call the first and second components of the bidegree of a homogeneous element $\gamma$ the \emph{weight} and \emph{degree} of $\gamma$ respectively, and denote the weight of $\gamma$ by $\wt(\gamma)$ and total degree by $|\gamma|$.

We will view the Ginzburg algebra as graded by degree (not total degree), where each degree homogeneous component is itself graded by weight. That is,
\begin{equation}\label{eqn:Grading}
\Gamma_Q=\bigoplus_{n\geq 0}(\Gamma_Q)_{*,n}
\end{equation} 
with $(\Gamma_Q)_{*,n}=\bigoplus_{k\geq 0}(\Gamma_Q)_{k,n}$ and $\Gamma_{k,n}$ is the subspace spanned by the paths in $\hater{Q}$ of weight $k$ and degree $n$.

The differential of $\Gamma_Q$ has bidegree $(0,-1)$ so the grading \eqref{eqn:Grading} descends to a grading in homology 
$$H_*\Gamma_Q=\bigoplus_{n\geq 0}H_{*,n}\Gamma_Q$$
where $H_{*,n}\Gamma_Q$ is the subspace of degree $n$ homology classes. It is itself graded by weight. Note that with this bigrading $$H_{*,0}\Gamma_Q=\k\barer{Q}/(\rho_i:i\in Q_0)=\Lambda_Q$$ so the homology $H_*\Gamma_Q$ contains the preprojective algebra in degree $0$ (\emph{cf.} Section \ref{sec:Dynkin}). 

We now turn to our main technical result:

\begin{theorem}\label{thm:GinzburgDerivedTranslationAlgebra}
For acyclic $Q$, there is a bigraded quasi-isomorphism 
\[
\Gamma_Q\to\bigoplus_{n\ge0}\Pdg(Q)(kQ,\tau^{-n}\k Q)
\]
inducing a bigraded $\k$-algebra isomorphism $H_*\Gamma_Q\xrightarrow{\sim} U_Q$. In particular, the $A_\infty$-structure on $H_*\Gamma_Q$ arising from Kadei\v{s}vili's Theorem is $A_\infty$-quasi-isomorphic to the $A_\infty$-structure on $U_Q$ arising from Proposition \ref{prop:U}.

\end{theorem}

Modulo the proof of Theorem \ref{thm:GinzburgDerivedTranslationAlgebra}, we can compute the minimal model of $\Gamma_Q$ for $Q$ a non-Dynkin quiver.

\begin{corollary}[See \cite{Keller3} Section 4.2]\label{cor:GinzburgNonDynkin}
If $Q$ is a non-Dynkin quiver, then the homology $H_*\Gamma_Q$ is isomorphic to the preprojective algebra $\Lambda_Q$ of $Q$, and $\Gamma_Q$ is a formal $A_\infty$-algebra.
\end{corollary}

\begin{proof}
By Theorem \ref{thm:GinzburgDerivedTranslationAlgebra} the homology $$H_*\Gamma_Q\cong U_Q=\bigoplus_{n\geq 0}\Db(Q)^\Z(\k Q,\tau^{-n}\k Q).$$ Since $Q$ is not Dynkin, $\tau^{-n}\k Q$ is quasi-isomorphic to a complex concentrated in degree 0, so $$\Db(Q)^\Z(\k Q,\tau^{-n}\k Q)=\Db(Q)(\k Q,\tau^{-n}\k Q)\cong\tau^{-n}\k Q$$ as $\k Q$-modules. Thus, $H_*\Gamma_Q$ is concentrated in degree 0 and hence is isomorphic to $\Lambda_Q$ as a graded algebra.

The $A_\infty$-structure maps $\mu_n$ have degree $n-2$, and hence are zero for $n\neq 2$, taking image in the 0 vector space.
\end{proof}

\subsection{Proof of Theorem \ref {thm:GinzburgDerivedTranslationAlgebra}.} 

Denote by $\cat G=\cat C_{\hater{Q}}$ the dg path category of $\hater{Q}$. The idea of the proof is to construct a Galois dg $\Z$-cover $\GG$ of $\cat G$ and a dg functor $R:\GG\to\Pdg(Q)$ inducing a quasi-isomorphism between $\k[\GG^S]\cong\Gamma_Q$ and $\bigoplus_{n\ge0}\Pdg(Q)(kQ,\tau^{-n}kQ)$, where $S$ is a section of the covering. The cover $\GG$ is closely related to the \emph{mesh category} of $Q$ whose construction we now recall.

\begin{definition}
The \emph{repetitive quiver} of an acyclic quiver $Q$ is the bigraded quiver $Q\times\Z$ with vertex set $Q_0\times\Z$ and two families of arrows:
\begin{enumerate}
\item For each $\alpha:i\to j$ in $Q$ and $n\in\Z$ there is an arrow $(\alpha,n):(i,n)\to (j,n)$ of bidegree $(0,0)$,
\item and for each $\alpha:i\to j$ in $Q$ and $n\in\Z$ there is an arrow $(\alpha^*,n):(j,n)\to (i,n+1)$ of bidegree $(1,0)$.
\end{enumerate}
\end{definition}

The first component of the bidegree will be referred to as the \emph{weight} and we denote the weight of an arrow $\gamma$ by $\wt(\gamma)$. The second component of the bidegree at the moment carries no additional information. In the sequel, we will construct honest bigraded quivers obtained from $Q\times\Z$ by adjoining extra arrows. We denote the total degree of an arrow $\gamma$ by $|\gamma|$; at the moment $|\gamma|=\wt(\gamma)$.  

The repetitive quiver has an (orientation preserving) automorphism $\tau$ defined by  
$$\tau(i,n)=(i,n+1) \qquad \text{and} \qquad \tau(\alpha,n)=(\alpha, n+1)$$
and an automorphism $\sigma$ of the arrow set uniquely determined by the property that $\sigma\gamma:\tau y\to x$ for an arrow $\gamma:x\to y$.  The automorphism $\tau$ is a bigraded automorphism of $Q\times\Z$, while $\sigma$ satisfies the identity $|\sigma\gamma|+|\gamma|=1$.

Let $\cat I_Q$ be the ideal of the path category $\cat C_{Q\times\Z}$ generated by the \emph{mesh relators}
$$\rho_x=\sum_{f:y\to x}(-1)^{|f|}f\circ (\sigma f)$$ for $x\in Q_0\times\Z$. The \emph{mesh category} of $Q$ is the quotient category $$\HH=\cat C_{Q\times\Z}/\cat I_Q.$$ 

\begin{theorem*}[Happel \cite{Happel2}]\label{thm:Happel}
There is a fully faithful functor 
$$h:\HH\hookrightarrow\ind\Db(Q)$$ 
where $\ind\Db(Q)$ is the full subcategory of $\Db(Q)$ obtained by choosing representatives from each isomorphism class of indecomposable objects. For every object $x$, the sequence
$$\xymatrix{h(\tau x)\ar[r] & \bigoplus_{f:y\to x} h(y) \ar[r] & h(x)\ar[r] & h(\tau x)[1]}$$
is a triangle. Moreover, $h$ is an equivalence of categories if and only if $Q$ is Dynkin, and in general, the image of $h$ is the transjective component of $\Db(Q)$, i.e., those objects $X$ such that $\tau^{n}X$ is projective for some $n\in\Z$.
\end{theorem*}

Up to isomorphism, the functor $h$ can be described explicitly by introducing coordinate functions on the objects of $Q\times\Z$. Denote by $q:Q_0\times\Z\to Q_0$ and $\ell:Q_0\times\Z\to\Z$ the projections onto the first and second coordinates respectively. With this notation, the functor $h:\HH\to\Db(Q)$ sends an object $x$ to $\tau^{\ell (x)}P_{q(x)}$.

We are now ready to construct the covering category $\GG$. Let $\hater{Q}\times\Z$ be the quiver obtained from the repetitive quiver $Q\times\Z$ by adjoining arrows $(t_i,n):(i,n)\to (i,n+1)$ of bidegree $(1,1)$. The automorphism $\tau$ of $Q\times\Z$ extends to $\hater{Q}\times\Z$ by setting $\tau(t_i,n)=(t_i,n+1)$. There is a unique degree $1$ arrow $t_x:x\to \tau x$ for each vertex $x$ of $\hater{Q}\times\Z$. The quiver $\hater{Q}\times\Z$ is a dg quiver with differential determined by 
\begin{align*}
d\gamma=& \ 0 & \text{for } & \gamma:x\to y  \text{ in } Q\times\Z \\
dt_x=& \ \rho_x=\sum_{\substack{\gamma:x\to y \\ \text{in } Q\times\Z }}(-1)^{|\gamma|}\gamma\cdot(\sigma \gamma) & \text{for } & t_x:x\to\tau x
\end{align*}
and the automorphism $\tau$ is a dg automorphism. Thus the path category $\GG=\cat C_{\hater{Q}\times\Z}$ is dg category, and $\tau$ endows $\GG$ with the structure of a dg $\Z$-category.

There is an evident dg quiver homomorphism $q:\hater{Q}\times\Z\to\hater{Q}$ by given by projection onto the first coordinate. It is a $\Z$-Galois covering of quivers, and hence induces a dg Galois $\Z$-covering functor $q:\GG\to\cat G$. The section $S$ of $q$ given by the objects $(i,0)$ of $\GG$ for $i\in Q_0$ induces an isomorphism $$\GG^S\xrightarrow{\sim}\cat G$$
by Proposition \ref{prop:SectionIsomorphism}. In the path category $\GG$, one has $\GG(x,y)=0$ if $\ell(x)>\ell(y)$, and so 
\begin{equation}\label{eqn:SectionalAlgebra}
\k[\GG^S]=\bigoplus_{\substack{x,y\in S\\ n\geq 0}}\GG(x,\tau^{-n}y).
\end{equation}

The inclusion of quivers $Q\hookrightarrow\hater{Q}\times\Z$ given by $i\mapsto (i,0)$ induces a functor $\iota:\cat C_Q\to\GG$. For an object $x$ of $\GG$ $$R(x)=\GG(-,x)\circ\iota:\cat C_Q\to\Cdg(\k)$$ is a dg $\cat C_Q$-module. This in turn defines a dg functor
\begin{gather*}
R:\GG\to\Cdg(Q) \\ x\mapsto R(x)
\end{gather*} 
into the category of dg $\cat C_Q$-modules. More concretely, for a vertex $j\in Q_0$, one has $R(x)(j)=\GG((j,0),x)$, which is bigraded and equipped with a degree $(0,-1)$ endomorphism $d$. We think of $R(x)(j)$ as a chain complex of graded vector spaces, where the internal degree is given by weight. Hence, the shift functor $[1]$ is taken with respect to the external grading only, and a homogeneous morphism $f$ skew-commutes with differentials in the sense that $d\circ f=(-1)^{|f|}f\circ d$.

Denote by $\GG^{-}$ the full subcategory of $\GG$ consisting of objects $x=(i,n)$ with $n \leq 0$.

\begin{proposition}\label{prop:ProjectiveResolution}
Let $x=(i,n)$ be an object of $\GG^{-}$. The complex $R(x)$ is a projective resolution of $\tau^nP_i$, and so in particular the image of $R$ lies in $\Pdg(Q)$.
\end{proposition}

\begin{proof} 
Let $\{f:x\to y\}$ be the set of irreducible degree 0 morphisms mapping to $x$. We first claim that $R(x)$ is isomorphic to the mapping cone of the morphism
$$\phi:R(\tau x)\to\bigoplus_{f:y\to x}R(y)$$
given by $\phi(s)=\sum_{f:y\to x}(\sigma f)\circ s$.

With our grading conventions, the mapping cone of $\phi$ is equipped with the differential 
$$d_{\Cone}=\begin{pmatrix} -(-1)^{|\sigma f|}d & \\ \phi & d\end{pmatrix}$$
on the $R(\tau x)[1]\oplus R(y)$ component. Here, the shift is taken with respect to degree.

We claim that 
\begin{gather*}\label{eqn:ProjectiveResolution}
\psi:\Cone(\phi)=R(\tau x)[1]\oplus\bigoplus_{f:y\to x}R(y)\to R(x) \\
(s,g)\mapsto t_x\circ s+\sum_{f:y\to x}f\circ g
\end{gather*}

is a chain isomorphism. Note it is an isomorphism of underlying graded spaces since any morphism with target $x$ factors through some $y$ or $\tau x$. So, we only need to verify that its components are chain maps.
 
On the one hand,
\begin{gather*}(\psi\circ d_{\Cone})(s,g)=\Phi(-(-1)^{|\sigma f|}ds, \sum_{f:y\to x}(\sigma f)\circ g+dg) \\
=-(-1)^{|\sigma f|}t_x\circ ds+\sum_{f:y\to x}\left(f\circ(\sigma f)\circ s +f\circ dg\right) \\
=(-1)^{|f|}t_x\circ ds+\sum_{f:y\to x}f\circ(\sigma f)\circ s +\sum_{f:y\to x}f\circ dg
\end{gather*}

Since $|f|+|\sigma f|=1$. On the other hand

\begin{gather*}
(d\circ\psi)(s,g)=d(t_x\circ s+\sum_{f:y\to x}f\circ g) \\
=dt_x\circ s+t_x\circ ds+\sum_{f:y\to x}(-1)^{|f|}f\circ dg \\
=\sum_{f:y\to x}(-1)^{|f|}f\circ (\sigma f)\circ s+t_x\circ ds+\sum_{f:y\to x}(-1)^{|f|}f\circ dg
\end{gather*}

and so $d_{\Cone}\circ\psi=(-1)^{|f|}\psi\circ d$.  Hence $\psi$ is a graded chain map since it's restriction to each $R(y)$ has degree $|f|$.

Let $L_j(x)$ denote the length of any degree 0 morphism $(j,0)\to x$ in $\GG$, and $L(x)=\min\set{L_j(x):j\in Q_0}$. We proceed by induction on $L(x)$. Note $L(x)=0$ implies $n=0$. Thus for any vertex $j\in Q_0$
$$R(x)(j)=\GG((j,0),(i,0))=C_Q(j,i)=P_i(j)$$ and so $R(x)=P_i$ is a projective resolution of $P_i$.  

Suppose now that $L(x)>0$. Since $L(y),L(\tau x)<L(x)$ we may assume by induction that $R(y)$ and $R(\tau x)$ are projective complexes quasi-isomorphic to $\tau^{\ell (y)}P_{q(y)}$ and $\tau^{n+1}P_i$ respectively. But $R(x)$ is quasi-isomorphic $\Cone(\phi)$ which is an acyclic projective complex. Moreover, $R(x)$ is quasi-isomorphic to the cokernel of 
$$\tau^{n+1}P_i\to\bigoplus_{f:y\to x}\tau^{\ell (y)}P_{q(y)}$$
and so $R(x)$ is a projective resolution of $\tau^nP_i$.  
\end{proof}

As a corollary we get that the complexes $R(\tau^-x)$ and $\tau^-R(x)$ are quasi-iso\-morphic for $x$ in $\GG^{-}$. Using this fact and the presentation of the sectional algebra from \eqref{eqn:SectionalAlgebra}, we get that the functor $R$ induces a dg algebra homomorphism
\begin{equation}\label{eqn:SectionalDerivedTranslationAlgebra}
\k[\GG^S]\to\bigoplus_{\substack{x,y\in S \\ n\geq 0}} \Pdg(Q)(R(x),\tau^{-n}R(y))
\end{equation}
where the algebra on the right hand side is quasi-isomorphic to the derived translation algebra $U_Q$.

The following Lemma originally appeared in \cite{IgusaTodorov}. We record the proof with slight modification to incorporate dg structures. 

\begin{lemma}\label{lem:ShortExactSequence}
Let $x$ be an object of $\GG$ and $\{f:y\to x\}$ the set of irreducible degree 0 morphisms mapping to $x$. Then there is a short exact sequence of functors of graded chain complexes
$$\xymatrix{0\ar[r] &\bigoplus_{f:y\to x} \GG(-,y)\ar[r]^{\phi} & \GG(-,x)\ar[r]^{\delta} & \GG(-,\tau x)[1]\ar[r] & 0}$$
where $[1]$ denotes shift with respect to the external (i.e. chain) grading.  
\end{lemma}

\begin{proof}
The map $\phi$ is defined on the component $\GG(z,y)$ by $\phi(g)=f\circ g$. It is a graded chain map since $d\circ\phi = (-1)^{|\phi|}\phi\circ d$.     

Any morphism $h:z\to x$ can be uniquely decomposed as  
\begin{equation}\label{eqn:Decomposition}
h=t_x\circ s+\sum_{f:y\to x}f\circ g
\end{equation}
for $s:z\to\tau x$ and $g:z\to y$. (Recall that the morphisms in $\GG$ go in the opposite direction of the arrows in $\hater{Q}\times\Z$.) Define $\delta(h)=s$, which is again a graded chain map: $$dh=t_x\circ ds+\sum_{f:y\to x}f\circ\left( \sigma f\circ s + (-1)^{|f|}dg\right)$$ and so $\delta(dh)=ds=d\delta(h)$. In other words, $\delta\circ d=(-1)^{|\delta|}d[1]\circ\delta$.

The uniqueness of the decomposition \eqref{eqn:Decomposition} implies that the sequence is exact.    
\end{proof}

We now show that the homomorphism \eqref{eqn:SectionalDerivedTranslationAlgebra} induces an isomorphism in homology, completing the proof of Theorem \ref{thm:GinzburgDerivedTranslationAlgebra}.

\begin{proposition}
The bigraded functor $$H_*R:H_*\GG^{-}\to H_*\Pdg(Q)$$ induced in homology is fully faithful.
\end{proposition}

\begin{proof}
Let $x,y$ be objects in $\GG^{-}$ and define $L(x,y)$ to be the minimal length among degree 0 morphisms $x\to y$ in $\GG$. We show that the (weight) graded map $$H_nR_{xy}:H_n\GG(x,y)\to \Db(Q)^\Z_n(R(x),R(y))$$ is an isomorphism for every $n$ by induction on $n$ and $L(x,y)$.

For $n=0$, one has $H_0\GG(x,y)=\HH(x,y)$ and $$\Db(Q)^\Z(R(x),R(y))=\Db(Q)(R(x),R(y)).$$ By Proposition \ref{prop:ProjectiveResolution}, $R(x)$ and $R(y)$ are quasi-isomorphic to $h(x)$ and $h(y)$ respectively, where $h:\HH\to\Db(Q)$ is Happel's functor (\emph{cf.} Theorem \ref{thm:Happel}). If $L(x,y)=1$, then there is exactly one morphism $x\to y$ in $\GG$, and moreover this morphism has weight 0. Thus the complexes $H_*\GG(x,y)$ and $\Db(Q)^\Z(R(x),R(y))$ are concentrated in degree 0, and hence isomorphic by Happel's Theorem.

Suppose now that $L(x,y)>1$, and consider the degree 0 irreducible morphisms $f:z\to y$. Note $L(x,\tau y),L(x,z)<L(x,y)$ since every degree 0 morphism $x\to y$ in $\GG$ factors through $\tau y$ or some such $z$. By Lemma \ref{lem:ShortExactSequence} there is a short exact sequence
$$\xymatrix{0\ar[r] &\bigoplus_{f:z\to y} \GG(x,z)\ar[r] & \GG(x,y)\ar[r] & \GG(x,\tau y)[1]\ar[r] & 0}$$ 
of complexes of graded spaces. It induces a long exact sequence in homology and so there is a commutative diagram
$${\footnotesize\xymatrix@C=12pt{H_n\GG(x,\tau y)\ar[r]\ar[d] & \bigoplus H_n\GG(x,z)\ar[r]\ar[d] & H_n\GG(x,y)\ar[r]\ar[d] & H_{n-1}\GG(x,\tau y)\ar[r]\ar[d] & \bigoplus H_{n-1}\GG(x,z)\ar[d] \\
 {\cat D}^\Z_n(Rx,\tau Ry)\ar[r] & \bigoplus{\cat D}^\Z_n(Rx,Rz)\ar[r]& {\cat D}^\Z_n(Rx,Ry)\ar[r] & {\cat D}^\Z_{n-1}(Rx,\tau Ry)\ar[r] & \bigoplus{\cat D}^\Z_{n-1}(Rx,Rz)}}$$
where the direct sums are taken over all degree 0 irreducible morphisms $f:z\to y$. The first, second, fourth, and fifth vertical maps are isomorphisms by induction, and so by the Five Lemma $H_n\GG(x,y)\to\Db(Q)^\Z(R(x),R(y))$ is an isomorphism.
\end{proof}


\section{The Dynkin case}\label{sec:Dynkin}

We now specialize to Dynkin quivers. Somewhat surprisingly, $\Gamma_Q$ has a richer $A_\infty$-structure in this setting. 

\subsection{The twisted polynomial algebra}\label{sec:TwistedAlgebra} Let $K$ be a bialgebra with comultiplication $\Delta$ and suppose $A$ is a (left) $K$-module algebra. That is, $A$ is equipped with an associative product $\mu:A\otimes A\to A$ and a left $K$-action satisfying
$$x\mu(a,b)=\sum\mu(x_{(1)}a,x_{(2)}b)$$ for all $a,b\in A$ and $x\in K$. Here we use the Sweedler notation $\Delta(x)=\sum x_{(1)}\otimes x_{(2)}$. The \emph{smash product} is the algebra $A\sharp K$ whose underlying $\k$-vector space is $A\otimes K$ with multiplication given by 
\begin{equation}\label{eqn:SmashProduct}
(a\otimes x)\cdot(b\otimes y)=\sum (a\cdot x_{(1)}b)\otimes x_{(2)}y
\end{equation}
for $a,b\in A$ and $x,y\in K$. One readily checks that this defines an associative multiplication making $A\sharp K$ into an associative $\k$-algebra.

If $x$ is a group-like element of $K$ in the sense that $\Delta(x)=x\otimes x$, then the product \eqref{eqn:SmashProduct} simplifies to $a(xb)\otimes xy$.

\begin{definition}
The polynomial ring $\k[t]$ is a bialgebra having $t$ as a group-like element for the comultiplication. If $A$ is an ordinary algebra and $\phi$ an algebra endomorphism of $A$, then $A$ is a module over the bialgebra $\k[t]$, where $t$ acts by $\phi$. We call the smash product $A\sharp\k[t]$ the \emph{$\phi$-twisted polynomial algebra} and denote it by $A^\phi[t]$.  
\end{definition}

Using a standard filtration argument, one obtains the following alternative description of the algebra $A^\phi[t]$.

\begin{lemma}\label{lem:TwistedAsTensors}
Let $A$ be a $\k$-algebra with automorphism $\phi$. Let $A[t]$ denote the vector space of polynomials with (left) coefficients in $A$, and let $R$ be the ideal of the tensor algebra $T(A[t])$ generated by the tensors $$at^p\otimes bt^q-a\phi^p(b)t^{p+q}$$ for $a,b\in A$. Then there is an isomorphism $$T(A[t])/R\xrightarrow{\sim} A^\phi[t]$$ of $\k$-algebras. \qed
\end{lemma}

Suppose that $A=\k Q/I$ is a bound path algebra and $\phi$ is an endomorphism of $A$ induced by an automorphism of $Q$ that preserves the ideal $I$. We want to describe the twisted polynomial algebra $A^\phi[t]$ as a bound path algebra. 

\begin{definition}
Define a quiver $\Omega$ by adjoining to $Q$ an arrow $u_i:\phi(i)\to i$ for each vertex $i\in Q_0$. Let $J$ be the ideal of $\k\Omega$ generated by $I$ and the elements $$\omega_{\alpha}=u_i\alpha-\phi(\alpha)u_j$$ for arrows $\alpha:i\to j$ in $Q$.
\end{definition}

\begin{proposition}\label{prop:TwistedAsPathAlg}
The twisted polynomial algebra $(\k Q/I)^\phi[t]$ is isomorphic to $\k\Omega/J$.
\end{proposition}

\begin{proof}
We first prove in the case $A=\k Q$, i.e. $I=0$. Let $\theta:k\Omega\to A^\phi[t]$ be the $\k$-algebra homomorphism determined by $\theta(u_i)=te_i$ and the requirement that the restriction of $\theta$ to $\k Q\subset\k\Omega$ is the identity. The relator $\omega_\alpha$ maps to
$$te_i\alpha-\phi(\alpha)te_j=t\alpha-t\alpha=0$$ under $\theta$. Hence $J\subset\ker\theta$, and so there is an induced algebra homomorphism $\bar{\theta}:\k\Omega/J\to A^\phi[t]$. We now provide an inverse to $\bar{\theta}$.

Let $R$ be the ideal of $T(A[t])$ as in Lemma \ref{lem:TwistedAsTensors} so that $A^\phi[t]\cong T(A[t])/R$. The map $A[t]\to\k\Omega$ evaluating a polynomial at $u$ determines a $\k$-algebra homomorphism $\psi:T(A[t])\to\k\Omega$. The homomorphism $\psi$ sends a generator $\alpha t^p\otimes \beta t^q-\alpha\phi^p(\beta)t^{p+q}$ of $R$ to $$\alpha u^p\beta u^q-\alpha\phi^p(\beta)u^{p+q}$$ which lies in $J$ by induction on $p$.  Hence, $\psi$ descends to a homomorphism $\bar{\psi}:T(A[t])/R\to\k\Omega/J$, which is readily checked to be inverse to $\bar{\theta}$.

Now suppose that $A=\k Q/I$ with $I$ possibly non-zero. We note that $A^\phi[t]\cong\k Q^\phi[t]/I[t]$ where $I[t]$ is the ideal of $\k Q_\phi[t]$ generated by polynomials with coefficients in $I$. The homomorphism $\k Q^\phi[t]\to\k\Omega/(\omega_\alpha:\alpha\in Q_1)$ sends the ideal $I[t]$ to $J/(\omega_\alpha)$, and so there are isomorphisms
$$A^\phi[t]\cong\frac{\k\Omega/(\omega_\alpha)}{J/(\omega_\alpha)}\cong\k\Omega/J.$$
\end{proof}

\subsection{Preprojective algebras}\label{sec:Preprojective} 

Of particular interest to us will be a twisted polynomial algebra with coefficients in the preprojective algebra of a Dynkin quiver $Q$, whose construction we now recall. Denote by $\barer{Q}$ the degree $0$ subquiver of $\hater{Q}$. The \emph{preprojective algebra} is the graded algebra $$\Lambda_Q=\k\barer{Q}/(\rho_i:i\in Q_0)$$ where $\rho_i$ are as in Section \ref{sec:GinzburgAlgebra}. Evidently $\Lambda_Q=H_{*,0}\Gamma_Q$, which we can think of as determining a bigrading on $\Lambda_Q$.

There is an isomorphism 
\begin{equation}\label{eqn:Preprojective}
\Lambda_Q\xrightarrow{\sim}\bigoplus_{n\geq 0}\Hom_{\k Q}(\k Q,\tau^{-n}\k Q)
\end{equation}
given by sending an arrow $\alpha:i\to j$ of $\barer{Q}$ to an irreducible morphism $P_j\to P_i$ (resp. $P_j\to\tau^{-}P_i)$ if $\alpha\in Q_1$ (resp. $\alpha^*\in Q_1$). The product on the right-hand side of is defined analogously to that of $U_Q$. The Nakayama functor $\nu(-)=D\Hom_{\k Q}(-,\k Q)$ and its inverse determine an involution $\nu$ of $\Lambda_Q$ under the isomorphism \eqref{eqn:Preprojective}.

From Proposition \ref{prop:TwistedAsPathAlg} the twisted polynomial algebra $\Lambda_Q^\nu[u]$ is isomorphic to the bound path algebra $\k\Omega/J$ where $\Omega$ is the quiver obtained from $\barer{Q}$ by adding arrows $u_i:\nu(i)\to i$ and $J=(\omega_\alpha,\rho_i:\alpha\in\barer{Q}_1, i\in Q_0)$. The involution $\nu$ naturally extends to $\k\Omega$ by defining $\nu(u_i)=u_{\nu(i)}$. It satisfies $\nu(\omega_\alpha)=\omega_{\nu(\alpha)}$ and so descends to an involution of $\k\Omega/J$.

\begin{example}
Consider the preprojective algebra of type $A_3$, with quiver
\[
\xymatrix{3\ar@/^/[r]^{\beta} & 2\ar@/^/[r]^{\alpha}\ar@/^/[l]^{\beta^*} & 1\ar@/^/[l]^{\alpha^*}}
\]
and relators $\rho_1=-\alpha\alpha^*$, $\rho_2=\alpha\alpha^*-\beta^*\beta$, and $\rho_3=\beta\beta^*$.
The involution $\nu$ sends $\nu(e_1)=e_3$, $\nu(e_2)=e_2$, and $\nu(\alpha)=\beta^*$; the rest of the action of $\nu$ is determined by the fact that it commutes with $(-)*$ and the involution property. 

Thus, the quiver $\Omega$ and generators of $J$ are given by

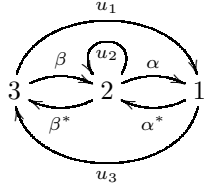
\begin{minipage}[t]{0.33\linewidth}
\[
\xymatrix{3\ar@/^/[r]^{\beta}\ar@(u,u)[rr]^{u_1} & 2\ar@/^/[r]^{\alpha}\ar@/^/[l]^{\beta^*}\ar@(ur,ul)^{u_2} & 1\ar@/^/[l]^{\alpha^*}\ar@(d,d)[ll]^{u_3}}
\]
\end{minipage}
\begin{minipage}[t]{0.67\linewidth}
\begin{gather*}
\omega_{\alpha}=u_2\alpha-\beta^* u_1 \qquad
\omega_{\beta}=u_3\beta-\alpha^*u_2 \\
\omega_{\alpha^*}=u_1\alpha^*-\beta u_2 \qquad \omega_{\beta^*}=u_2\beta^*-\alpha u_3 \\
\rho_1=-\alpha\alpha^* \qquad 
\rho_2=\alpha\alpha^*-\beta^*\beta \qquad
\rho_3=\beta\beta^*.
\end{gather*}
\end{minipage}
\end{example}

We now extend the bigrading of $\barer{Q}\hookrightarrow\hater{Q}$ to a bigrading of $\Omega$ in the a priori bizarre manner as follows:  
Define a function $N:Q_0\to\N$ determined by the condition $$\tau^{-N(i)}P_{\nu(i)}=P_i[1]$$ where $P_i$ is the indecomposable projective $kQ$-module corresponding to the vertex $i$. The bidegree of the arrow $u_i$ is defined to be $(N(i),1)$.

\begin{lemma}
The relators $\omega_\alpha$ are homogeneous with respect to this bigrading, and so the bigrading of $\k\Omega$ descends to $\Lambda^\nu[u]=\k\Omega/J$.
\end{lemma}

\begin{proof}
It is clear that $\omega_\alpha$ is homogeneous with respect to degree so it suffices to show that $N(i)+\wt(\alpha)=N(j)+\wt(\nu(\alpha))$. We will show $N(i)-N(j)=\wt(\nu(\alpha))-\wt(\alpha)$ by a case-by-case analysis of the weights of $\alpha$ and $\nu(\alpha)$. If $\wt(\alpha)=0$ and $\nu(\alpha)=1$ the arrows $\alpha:i\to j$ and $\nu(\alpha)^*:\nu(j)\to\nu(i)$ are in $Q$, and so there are irreducible morphisms $$P_j\to P_i \ \quad \ \text{and} \ \quad \ P_{\nu(i)}\to P_{\nu(j)}$$ in $\mod\text{-}\k Q$. Hence in the Auslander-Reiten quiver of $\Db(Q)$ there is a subquiver of the form 
$$
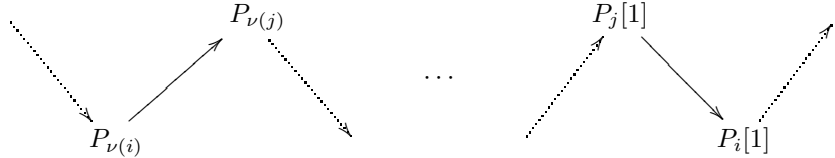
\xymatrix@R=1em@C=2em{\ar@{..>}[ddr] & & \text{ \ } P_{\nu(j)}\ar@{..>}[ddr] & & & & P_j[1]\ar[rdd] & & \\
& & & & \cdots & & & & \\
& \text{ \ } P_{\nu(i)}\ar[uur] & & & & \ar@{..>}[uur] & & P_i[1]\ar@{..>}[uur] & }$$ so $N(i)-N(j)=1=\wt(\nu(\alpha))-\wt(\alpha)$. The other three cases for the weights of $\alpha$ and $\nu(\alpha)$ are similar.
\end{proof}

\subsection{The minimal model in Dynkin type}

We are now ready to prove the main theorem computing the minimal model of $\Gamma_Q$ in Dynkin type.

\begin{theorem}\label{thm:TwistedDerivedTranslationAlgebra}
Suppose $Q$ is Dynkin. Then there is a bigraded $\k$-algebra isomorphism 
$$\Lambda_Q^\nu[u]\xrightarrow{\sim} U_Q$$ sending $u_i$ to $s_i:P_i\to P_i[1]$.
\end{theorem}

Combining with Theorems \ref{thm:GinzburgDerivedTranslationAlgebra} and \ref{thm:DerivedCategoryAinfty} we get the following Corollary:

\begin{corollary}
If $Q$ is Dynkin, the twisted polynomial algebra $\Lambda_Q^\nu[u]$ admits a minimal $A_\infty$-structure so that:
\begin{enumerate}
\item the maps $\mu_n$ are $u$-equivariant,
\item $\Lambda_Q^\nu[u]$ is generated as an $A_\infty$-algebra by $\Lambda$,
\item $\mu_n=0$ for $n\neq 2,3$.
\end{enumerate}
Moreover, $\Lambda_Q^\nu[u]$ with this $A_\infty$-structure is a minimal model for the Ginzburg algebra $\Gamma_Q$.
\end{corollary}

The strategy to prove the Theorem is analogous to that used to prove Theorem \ref{thm:GinzburgDerivedTranslationAlgebra}. We construct a Galois $\Z$-cover $\OO$ of the path category $\cat O:=\cat C_{\Omega}/\cat J$ of $\Lambda^\nu[u]=\k\Omega/J$ and a functor $\rr:\tilder{\cat O}\to\Db(Q)^\Z$ inducing an isomorphism between $\k[\OO^S]\cong\Lambda_Q^\nu[u]$ and $U_Q$ for some section $S$ of the covering.

Let $\tilder{\Omega}=\Omega\times\Z$ be obtained from the repetitive quiver $Q\times\Z$ by adjoining arrows $(u_i,n):(\nu(i),n-N(i))\to (i,n)$ of bidegree $(N(i),1)$. The automorphism $\nu$ of $\Omega$ lifts to $\tilder{\Omega}$ by setting 
$$\nu (i,n)=(\nu(i),n-N(i)) \ \quad \ \text{and} \ \quad \ \nu(\gamma, n)=(\nu(\gamma), n-N(s\gamma))$$ for $\gamma$ an arrow of $\Omega$. Note that for every $x$ there is a unique degree $1$ morphism $u_x:x\to\nu(x)$ in the path category $\tilder{\cat C}_\Omega:=\cat C_{\tilder{\Omega}}$ provided by the arrow $(u_i,n):\nu(i,n)\to (i,n)$.

The automorphism $\tau$ of $Q\times\Z$ extends to $\tilder{\Omega}$ in the evident manner, and commutes with the automorphism $\nu$. This defines a $\nu$-equivariant $\Z$-action on the path category $\tilder{\cat C}_{\Omega}$.
There is an evident covering morphism $q:\tilder{\Omega}\to\Omega$ given by projection onto the first factor which induces a $\nu$-equivariant Galois $\Z$-covering $q:\tilder{\cat C}_\Omega\to\cat C_{\Omega}$. The section $S$ of $q$ given by the objects $(i,0)$ for $i\in Q_0$ induces a $\nu$-equivariant isomorphism
$$\tilder{\cat C}_\Omega^S\xrightarrow{\sim}\cat C_{\Omega}.$$

Let $\tilder{\cat J}$ be the ideal of the category $\tilder{\cat C}_\Omega$ generated by the morphisms
$$\rho_x=\sum_{\substack{f:y\to x \\ \text{in } Q\times\Z}}(-1)^{|f|}f\circ (\sigma f)  \ \quad \ \text{and} \ \quad \ \omega_f=u_y\circ f-\nu(f)\circ u_x$$ for objects $x$ and degree 0 irreducible morphisms $f:x\to y$, and set $\OO=\tilder{\cat C}_\Omega/\tilder{\cat J}$.

The ideal $\tilder{\cat J}$ is preserved by the automorphisms $\tau$ and $\nu$, and the restriction of the coving $q:\tilder{\cat C}_\Omega\to\cat C_{\Omega}$ to $\tilder{\cat J}$ is a covering of the ideal $\cat J$ of $\cat C_\Omega$. Thus by Lemma \ref{lem:CoveringIdeals} there is an induced Galois $\Z$-covering $\OO\to\cat O$, and hence an algebra isomorphism $$\k[\OO^S]\xrightarrow{\sim}\k[\cat O]$$ by Proposition \ref{prop:SectionIsomorphism}.

The inclusion of quivers $Q\hookrightarrow\tilder{\Omega}$ induces an embedding of path categories $\iota:\cat C_Q\to\OO$. For an object $x$ of $\OO$, there is a contravariant functor $$\rr(x)=\OO(-,x)\circ\iota:\cat C_Q\to\Cgr(\k)$$ which we think of as a bigraded $\k Q$-module. This in turn defines a functor 
\begin{gather*}
\rr:\OO\to\Db(Q)^\Z \\
x\mapsto \rr(x)
\end{gather*}
with image in the augmented derived category, since $\OO$ has morphisms of non-zero degree. Denote by $\tilder{\cat C}_\Omega^{-}$ and $\OO^{-}$ the full subcategories of $\tilder{\cat C}_\Omega$ and $\OO$ whose objects $x=(i,n)$ have $n\leq 0$. The functor $\nu$ preserves both $\tilder{\cat C}_\Omega^{-}$ and $\OO^{-}$, but it's inverse does not.

\begin{lemma}
Suppose $x$ is an object of $\OO^{-}$. Composition with $u_x$ induces an isomorphism of graded functors
$$\OO^{-}(-,x)[1]\xrightarrow{\sim}\OO^{-}(-,\nu(x))$$
where shift is taken with respect to degree. Hence, the $\k Q$-modules $\rr(\nu(x))$ and $\rr(x) [1]$ are isomorphic.
\end{lemma}

\begin{proof}
Composition with $u_x$ induces a natural transformation $$\tilder{\cat C}_\Omega(-,x)[1]\to\tilder{\cat C}_\Omega(-,\nu(x))$$ which when restricted to $\tilder{\cat C}_\Omega^{-}$ sends the subfunctor $\tilder{\cat J}(-,x)[1]$ surjectively onto $\tilder{\cat J}(-,\nu(x))$. Hence it induces a natural transformation of the Lemma.

For any $z$ in $\OO^-$ the restriction $\OO_0^{-}(z,x)\to\OO_1^{-}(z,\nu(x))$ is an isomorphism: any degree $1$ morphism $z\to\nu(x)$ factors through a single $u_y:y\to \nu(y)$ modulo $\tilder{\cat I}$ and so uniquely factors through $u_x:x\to\nu(x)$ modulo $\tilder{\cat J}$. The proof is completed by induction on degree.
\end{proof}

The mesh category $\HH$ is manifestly a $\nu$-invariant subcategory of $\OO$ whose morphisms are the degree 0 morphisms of $\OO$. Denote by $\HH^{-}$ the corresponding subcategory of $\OO^{-}$.

\begin{lemma}\label{lem:RepresentationVsHappel}
The restriction of $\rr$ to $\HH^{-}$ is naturally equivalent to Happel's functor $h$. 
\end{lemma}

\begin{proof}
The full subcategory of $\HH^{-}$ whose objects are the $x=(i,n)$ with $0\geq n>-N(\nu(i))$ is equivalent to $\mod\text{-}\k Q$ where the restriction of $h$ is equivalent to $\HH(-,x)\circ\iota$ for $\iota:\cat C_Q\to\HH$ the embedding induced by $Q\hookrightarrow Q\times\Z$. For such $x$ 
$$\HH(-,x)=\OO(-,x)$$
since the morphisms of the form $u_x$ map to objects $y=(j,m)$ with $m\leq -N(\nu(j))$.   

Given arbitrary $x=(n,i)$ in $\HH^{-}$ one has $\nu(x)=(\nu(i), n-N(i))$, and so 
$$h(\nu(x))=\tau^{n-N(i)}P_{\nu(i)}=\tau^nP_i[1]=h(x)[1].$$ 
Thus the Proposition holds by the previous Lemma and induction.
\end{proof}

The functors $\tau$ and $\rr$ commute up to isomorphism, and hence $\rr$ induces a bigraded $\k$-algebra homomorphism 
\begin{equation}\label{eqn:TwistedSection}
\k[\OO^S]\to\bigoplus_{\substack{x,y\in S \\ n\geq 0}}\Db(Q)^\Z(\rr(x),\tau^{-n}\rr(y))=U_Q.
\end{equation}

\begin{proposition}
The functor $\rr:\OO^{-}\to\Db(Q)^\Z$ is fully faithful, and so \eqref{eqn:TwistedSection} is an isomorphism.
\end{proposition}

\begin{proof}
Recall that the morphisms in $\Db(Q)^\Z$ are generated by the morphisms in $\Db(Q)$ together with the degree $1$ suspension morphisms $s_X:X\to X[1]$. By Proposition \ref{lem:RepresentationVsHappel}, the induced map $\rr_{xy}$ maps onto $\Db(Q)$. Moreover, $\rr(u_x)=s_{\rr(x)}$ and so $\rr$ is full. 

Suppose now that $f:\rr(x)\to\rr(y)$ is a homogeneous morphism in $\Db(Q)^\Z$. It can be uniquely written as $f=g\circ s^n_X$ for some degree 0 morphism $g:X[n]\to Y$. Thus, there is a unique $g'$ mapping to $g$ under Happel's functor $h$, which is equivalent to $E$ on degree 0 morphisms. Thus $g'\circ u_x^n$ is the unique morphism of $\OO^-$ mapping to $f$ under $\rr_{xy}$.
\end{proof}

\begin{example}
We illustrate how the results of this paper can be used to compute the higher multiplications of $\Lambda_Q^\nu[u]$ in the example where $Q$ is the quiver $\xymatrix{3\ar[r]^{\beta} & 2\ar[r]^{\alpha} & 1}$. Let us show that $\mu_3(\beta\alpha,\alpha^*,\alpha)=e_3ue_1.$ Similar calculations show that $\mu_3(\alpha,\alpha^*\beta^*,\beta)=e_2ue_2$ and $\mu_3(\alpha^*,\alpha,\alpha^*\beta^*)=e_1ue_3$.

The category $\Db(Q)$ has irreducible morphisms 
$$f:P_1\to P_2 \qquad g:P_2\to P_3 \qquad f^*:P_2\to S_2 \qquad g^*:P_3\to I_2$$
and all other irreducible morphisms are $\tau^-$ translates of these. Denote by $\phi:\Lambda_Q\xrightarrow{\sim} U_Q$ the isomorphism of Theorem \ref{thm:GinzburgDerivedTranslationAlgebra}. It sends $\alpha$, $\beta$, $\alpha^*$ and $\beta^*$ to $f$, $g$, $f^*$ and $g^*$ respectively.

By Theorem \ref{thm:TwistedDerivedTranslationAlgebra} we have 
$$\phi(\mu_3(\beta\alpha,\alpha^*,\alpha))=\mu_3(\phi(\beta\alpha),\phi(\alpha^*),\phi(\alpha))=\mu_3(g\circ f,f^*,f)$$
which we claim equals the morphism $s_1:P_1\to P_1[1]=\tau^-P_3$. Indeed, in terms of $\Db(Q)$ the right hand side is $\mu_3$ applied to the non-split triangle 
$$\xymatrix@R=10pt{
P_1\ar[r]^{f} & P_2\ar[r]^{f^*} & \tau^-P_1\ar[r]^{\tau^-(g\circ f)} & \tau^-P_3}$$ and hence equals $s_1$ by Theorem \ref{thm:DerivedCategoryAinfty}. Since $s_1:P_1\to\tau^-P_3$ it follows that $\phi^{-1}(s_1)=e_3ue_1$ and so $\mu_3(\beta\alpha,\alpha^*,\alpha)=e_3ue_1$.
\end{example}

\section*{Acknowledgements} 
The research presented here was completed as part of the author's Ph.D. thesis at Brandeis University under the supervision of Kiyoshi Igusa, and was partially supported by NSA grant \#H98230-13-1-0247. The author would like to express his gratitude to Kiyoshi for so generously sharing his ideas and insights, as well as his patient guidance over the years. He additionally thanks Claire Amiot, Jonah Ostroff, Hugh Thomas, and Gordana Todorov for helpful conversations and correspondences.


\bibliographystyle{amsplain}
\bibliography{thesisBib}

\end{document}